\documentclass[12pt]{article}
\usepackage{amsfonts}
\usepackage{amsmath, amssymb,amsthm}
\usepackage{tikz-cd}
\usepackage{indentfirst}
\usepackage{cite}
\usepackage{geometry}
\usepackage{amsthm}
\usepackage{graphicx}
\usepackage{float}
\usepackage{subfigure}
\usepackage{graphicx}
\usepackage{latexsym}
\usepackage[hidelinks]{hyperref}
\usepackage{titlesec}
\usepackage{lmodern}    
\usepackage{ifthen}   
\usepackage{lipsum}
\usepackage{authblk}
\usepackage{todonotes}
\newtheorem{thm}{Theorem}[section]
\newtheorem{lem}[thm]{Lemma}
\newtheorem{prop}[thm]{Proposition}
\theoremstyle{definition}
\newtheorem{definition}[thm]{Definition}
\newtheorem{ex}[thm]{Example}
\theoremstyle{definition}

\theoremstyle{definition}
\newtheorem{rk}[thm]{Remark}
\newtheorem{cor}[thm]{Corollary}
\newtheorem{def-thm}[thm]{Definition-Theorem}
\newtheorem{def-lem}[thm]{Definition-Lemma}
\title{A Construction of Vertex Algebra Bundles on Logarithmic Smooth Curves}

\newcommand{\LS}{\operatorname{LS}^{\operatorname{fs}}}

\newcommand{\Spec}{\operatorname{Spec}}

\newcommand{\Spf}
{\operatorname{Spf}}
\newcommand{\Res}{\operatorname{Res}}
\newcommand{\im}{\operatorname{Im}}
\newcommand{\Fra}{\underline{\operatorname{Frame}}}
\newcommand{\Aut}{\operatorname{Aut}}
\newcommand{\Autu}{\underline{\operatorname{Aut}}}
\newcommand{\Lie}{\operatorname{Lie}}

\newcommand{\Der}{\operatorname{Der}}

\newcommand{\End}{\operatorname{End}}
\newcommand{\Hom}{\operatorname{Hom}}

\newcommand{\id}{\operatorname{id}}

\newcommand{\gp}{\operatorname{gp}}

\newcommand{\0}{|0\rangle}

\newcommand{\vp}{\varphi}

\newcommand{\dlim}{\varinjlim}
\newcommand{\ul}{\underline}

\newcommand{\fet}{\textrm{f\'{e}t}}
\newcommand{\et}{\textrm{\'{e}t}}
\newcommand{\circX}{{\mathop{X}\limits^{\circ}}}

\newcommand{\cC}{\mathcal{C}}

\newcommand{\cF}{\mathcal{F}}

\newcommand{\cI}{\mathcal{I}}

\newcommand{\cL}{\mathcal{L}}
\newcommand{\cM}{\mathcal{M}}

\newcommand{\cO}{\mathcal{O}}

\newcommand{\cV}{\mathcal{V}}
\newcommand{\cW}{\mathcal{W}}

\newcommand{\cY}{\mathcal{Y}}

\newcommand{\bA}{\mathbb{A}}

\newcommand{\bC}{\mathbb{C}}

\newcommand{\bF}{\mathbb{F}}

\newcommand{\bN}{\mathbb{N}}

\newcommand{\bP}{\mathbb{P}}

\newcommand{\bR}{\mathbb{R}}

\newcommand{\bZ}{\mathbb{Z}}

\makeatletter
\def\blfootnote{\gdef\@thefnmark{}\@footnotetext}
\makeatother

\begin{document}
	\date{}
	\author{Xi-Chuan Tan}
	\affil{Department of Mathematics, University of Tsukuba}
	\maketitle
	
	\blfootnote{This work was supported by JST SPRING, Grant Number JPMJSP2124.}
	
	\begin{abstract}
		We present a construction of vertex algebra bundles and spaces of conformal blocks over families of logarithmic smooth curves.
		This work generalizes some earlier results by Frenkel and Ben-Zvi on vertex algebra bundles over complex smooth algebraic curves.
		We establish a weaker version of the propagation of vacua, 
		and compute the space of conformal blocks over a typical example of a nodal curve.
	\end{abstract}

	\section*{Introduction}
	Frenkel and Ben-Zvi developed a theory of vertex algebra bundles 
	in \cite{frenkel-benzvi} to study representation theory and algebraic curves. 
	In the present paper, 
	we generalize the base scheme
	of the bundle from a complex smooth algebraic curve  to a family of logarithmic smooth curves.
	This also provides an alternative approach to the vertex algebra theory on stable curves,
	compared with
	an earlier study by Damiolini, Gibney and Tarasca in \cite{damiolini}.

	There are three major ingredients.
	The first is to establish an appropriate definition of coordinates
	on logarithmic smooth curves.
	For a smooth complex curve $X$,
	a choice of an isomorphism of complete local rings
	$\widehat\cO_x\simeq\bC[[z]]$ is said to be a local coordinate at $x$.
	However such an isomorphism does not necessarily exist for a logarithmic smooth curve.
	Fortunately due to an earlier result of Kato in \cite{fkato} which asserts that a logarithmic smooth curve has at most ordinary double points,
	the construction of coordinates is reduced to
	the case that nodes appear.
	Motivated by the classical theory,
	we define a local coordinate to be a formally étale morphism of logarithmic schemes from the abstract disc $D$ to
	the curve $X$,
	which can be intuitively understood as an infinitesimal arc on the curve.
	We show that our definition is compatible with the classical one
	and works well with nodes.
	Roughly speaking,
	the reason for applying logarithmic geometry is that
	as a nice generalization of algebraic geometry in the language of schemes,
	it allows a stable curve to be considered smooth.
	In fact,
	smoothness is a notion depending on the category itself,
	and the category of log schemes admits more objects to be smooth.
	
	The second ingredient is the construction of the sheaf of vertex algebras associated to a given vertex algebra $V$ on the logarithmic smooth curve $X$.
	The group $\Autu^0\cO$ of continuous automorphisms of the disc $D$ has a natural action on the space $\Aut_x$ of local coordinate system at $x\in X$. 
	Moreover, 
	the conformal vertex algebra 
	(i.e., a vertex operator algebra in some other texts) 
	$V$ admits an $\Autu^0\cO$-action as well.  
	Then the associated sheaf $\cV$ 
	of vertex algebras
	is defined to be the ${\Autu^0\cO}$-equivariants of $V\otimes\pi_*\cO_{\Fra_X^\infty}$, 
	where $\Fra_X^\infty$ denotes the space of coordinates on $X$. 
	
	The third ingredient is the construction of conformal blocks.
	Conformal block is a concept from conformal field theory.
	Following the guidance of Carnahan in  \cite{carnahan2015equivariant}, 
	motivated by physical considerations
	we remove several points on $X$ to obtain the punctured space $\circX$. 
	For each conformal vertex algebra $V$, 
	there exists a space $\Lie_\circX(V)$ associated to $\circX$ which naturally acts on the $V$-module bundles. 
	If we only remove a single point $x$, 
	then the $\Lie_\circX(V)$-action
	is induced by
	the image of 
	the Lie algebra
	$\Lie_{D^\times}(V)$
	under the generalized vertex operation $\cY_{M,x}^\vee$
	where $D^\times$ is the punctured disc centered at $x$.
	Then the space of conformal blocks associated to the given data is defined to be
	$C_V(X,x,M)=\Hom_{\Lie_\circX(V)}(\cM,\cO_{\text{base}})$.
	As we will see in the main text,
	the spaces
	$\Lie_\circX(V)$ and
	$\Lie_{D^\times}(V)$
	heavily depend on the meromorphic differential forms on $X$.
	Thus the main difficulty is that
	it is hard to investigate these differential forms 
	near nodes.
	An approach is that we
	may apply the
	logarithmic differential forms
	which carry nice properties near nodes,
	rather than the usual meromorphic ones.
	More generally, 
	for an $n$-tuple of conformal $V$-modules $M_1,...,M_n$ respectively assigned to $n$ framed points in $X$, 
	the space of conformal blocks associated to the above data is $$C_V(X,(x_i),(M_i))=\Hom_{\Lie_\circX(V)}(\bigotimes_i \cM_i,\cO_{\text{base}}).$$ 
	There is a natural way to extend our construction to the multiple points case.
	We present a weak version propagation of vacua,
	which asserts that
	$$
	C_V(X,(x_i,y),(M_i,V))
	\hookrightarrow
	C_V(X,(x_i),(M_i)).
	$$
	This intuitively tells us that 
	if we puncture one more point on the curve and attach a regular $V$-module at this point,
	then the associated space of conformal blocks will be a subspace of the original one.

	The paper is organized as follows. 
	First we devote a section to reviewing log geometry and vertex algebras. 
	Then we  introduce the coordinate system, which is the key issue in the present paper. Constructions and several properties of vertex algebra bundles and conformal blocks are given in the sequel. 
	Finally we calculate a standard example from log geometry to illustrate our theory.

	\subsection*{Notation}
	For a sheaf $\cF$ on a Grothendieck topology
	$\cC=(\operatorname{Cat}\cC,\operatorname{Cov}\cC)$,
	we denote by $\cF(U')$
	the value of $\cF$ at
	$\{U'\to U\}\in\operatorname{Cov}\cC$.
	The symbol $\Gamma(U',\cF)$
	will be used with a different meaning in Section \ref{vacb},
	and so we would like to notice this to avoid confusion.
	If $X$ and $Y$ are objects in $\operatorname{Cat}\cC$,
	we set 
	$X(Y)=\Hom_{\operatorname{Cat}\cC}(Y,X)$,
	i.e., the collection of $Y$-points in $X$.
	
	\subsection*{Acknowledgments}
	The author would like to thank Scott Carnahan for many helpful comments.

	\section{Preliminaries}
	
	\subsection{Logarithmic geometry}\label{logsch}
	
	In this section we recall basic notions in logarithmic geometry.
	For the sake of simplicity, we write log (resp. prelog) for logarithmic (resp. prelogarithmic). 
	The theory of logarithmic geometry was initially established  by Fontaine, Illusie, and Kato  \cite{kato}.
	Details and further topics about log schemes can be found in \cite{ogus}.
	
	\begin{definition}
		Let $X$ be a scheme.
		A \textit{log structure} on $X$ is a morphism of sheaves of monoids $\alpha:M_X\to\cO_X$ such that the restriction of $\alpha$ to $\alpha^{-1}(\cO^*_X)$ is an isomorphism.
		A \textit{log scheme} is a pair $(X,\alpha)$
		consisting of a scheme $X$
		and a log structure 
		$\alpha$ on it.
	\end{definition}

	We note that the monoid structure on a commutative ring in the present paper is always given by multiplication. 
	Throughout the paper, a log structure on a scheme $X$ means a log structure on the small \'{e}tale site $X_\et$ unless specified. 
	For ease of notation, 
	we simply write $X$ for a log scheme,
	and denote by $\ul X$ the underlying scheme of $X$,
	with a log structure $\alpha_X$ on $\ul{X}$.
	And we still use $\cO_X$ to denote the structure sheaf of the underlying scheme of $X$.
	A morphism of log structures on $X$ from $\alpha:M_X\to \cO_X$
	to
	$\alpha':M_X'\to\cO_X$
	is a morphism of sheaves of monoids
	$\theta:M\to M'$
	such that 
	$\alpha'\circ\theta=\alpha$.
	Then log structures and morphisms between them then form the category of log structures on a given scheme.
	
	\begin{definition}
		A \textit{morphism of log schemes} $f:X\to Y$
		is 
		a morphism of the underlying schemes
		$\ul{f}:\ul{X}\to\ul{Y}$
		together with the following commutative diagram
		of sheaves of monoids on $Y$:
		$$
		\begin{tikzcd}
			M_Y \arrow[d, "\alpha_Y"'] \arrow[r, "f^\flat"] & f_*M_X \arrow[d, "f_*\alpha_X"] \\
			\cO_Y \arrow[r, "f^\sharp"]                     & f_*\cO_X                       
		\end{tikzcd}
	    $$
	\end{definition}
	
	\begin{definition}
		A \textit{prelog structure} $\alpha:M_X\to\cO_X$ on a scheme $X$ is a morphism of sheaves of monoids on $X_\et$.
		Likewise,
		a \textit{prelog scheme} is a pair $(X,\alpha)$
		consisting of a scheme $X$
		and a prelog structure 
		$\alpha$ on it.
	\end{definition}
    
    Every prelog structure
    $\alpha:M_X\to\cO_X$ on $X$
    admits an associated log structure
    $\alpha^{\log} :M^{\log}_X \to\cO_X$,
    defined to be the push-out of
    $$
    \begin{tikzcd}
    		M_X & \alpha^{-1}(\mathcal{O}_X^*) \arrow[r, "\alpha|_{\alpha^{-1}(\mathcal{O}_X^*)}"] \arrow[l, '] & \mathcal{O}_X^*
    \end{tikzcd}
    $$
    in the category of sheaves of monoids on $X_\et$, endowed with 
    $$M^{\log}_X \to\cO_X,\quad (a,b)\mapsto \alpha(a)b
    \quad
    \text{for}\ a\in M_X,\ b\in\cO^*_X.$$
    
    Analogously to the scheme-commutative ring
    dictionary,
    we introduce the
    log scheme-log ring
    dictionary as follows.
    \begin{definition}
    	A \textit{log ring} is a morphism of monoids
    	$\beta:P\to A$
    	where $A$ is a commutative ring.
    \end{definition}
    A log ring $\beta:P\to A$ is sometimes denoted by $(A,P)$ or $(A,\beta)$
    if no confusion arises. 
    If $P\to A$ is a log ring, then $\Spec(P\to A)$ is defined to be the log scheme whose underlying scheme is $\ul X:=\Spec A$, together with the log structure $P^{\log}\to \cO_X$ induced by $P\to \cO_X$ where $P$ is viewed as a constant sheaf on $X_\et$. 
    By a morphism $f:(A,P)\to (B,Q)$ of log rings, 
    we mean the following commutative diagram
    $$
    \begin{tikzcd}
    	A \arrow[r, "f^\sharp"]          & B           \\
    	P \arrow[u] \arrow[r, "f^\flat"] & Q \arrow[u]
    \end{tikzcd}
    $$
    of monoids.
    The above morphism $f$ is called an \textit{isomorphism} if $f^\sharp$ is an isomorphism of rings and $f^\flat$ is strict,
    i.e., the morphism 
    $P/P^*\to Q/Q^*$ induced by $f^\flat$ is an isomorphism.
    
    A monoid $M$ is called \textit{integral} if $m,m',m''\in M$ and $m+m'=m+m''$ implies that $m'=m''$. 
    If in addition $M$ is finitely generated, we say $M$ is \textit{fine}. 
    A monoid $Q$ is said to be \textit{saturated} if it is integral and if whenever $q\in Q^{\operatorname{gp}}$ is such that $mq\in Q$ for some $m\in\bZ_{>0}$, then $q\in Q$.
    Here $Q^{\operatorname{gp}}$ is the group associated to $Q$ which is identified with the cokernel of the diagonal embedding
    $Q\to Q\oplus Q$.
    
    A \textit{chart} for a log scheme $(X,M_X)$ subordinate to $Q$ is a monoid homomorphism $\beta:Q\to M_X(X)$ such that the associated morphism of sheaves of monoids $Q^{\log}\to M_X$ is an isomorphism.
	We call $M_X$ \textit{quasi-coherent} (resp. \textit{coherent}) if the restriction of $M_X$  to any $U\in X_\et$ admits a chart (resp. a chart subordinate to a finitely generated monoid). 
	A sheaf of monoids is \textit{fine} if it is coherent and integral. If in addition the domains of all these charts are saturated, it is said to be \textit{fs}. 
	A log scheme $(X,M_X)$ is called \textit{fs} (resp. coherent) if $M_X$ is.
    
    \begin{rk}
    	The inclusion $\cO_X^*\to\cO_X$
    	is called the 
    	\textit{trivial log structure}.
    	In particular,
    	the trivial log structure
    	on an affine open subset
    	$\Spec A$
    	is modeled by the monoid morphism
    	$0\to A$.
    \end{rk}
    
    \begin{definition}[\textbf{direct image of log structure}, \cite{ogus}-III-1.1.5]\label{direct}
    	Let $f:X'\to X$ be a morphism of schemes.
    	If $\alpha':M_{X'}\to\cO_{X'}$
    	is a log structure on $X'$, 
    	consider the Cartesian diagram
    	$$
    	\begin{tikzcd}
    		f_*(M_{X'})^{\log} \arrow[d] \arrow[r, "f_*(\alpha')^{\log}"] & \cO_X \arrow[d, "f^\sharp"] \\
    		f_*(M_{X'}) \arrow[r, "f_*(\alpha')"]                         & f_*(\cO_{X'})              
    	\end{tikzcd}
        $$
        Then  
        $f_*(\alpha')^{\log}$
        is a log structure on $X$,
        and
        $f_*^{\log}$
        is a functor from the category of log structures on $X'$ to the category of log structures on $X$.
    \end{definition}

    \begin{rk}
    	We now describe how a log ring
    	$\beta: P\to A$ gives rise to a log scheme
    	$X=\Spec(P \to A)$
    	modeled by it.
    	The monoid $P$ is viewed as a constant sheaf over $\Spec A$,
    	and this induces a prelog structure
    	$\beta:P\to\cO_X$.
    	For an open subset $U$ of $X$,
    	the value of $P^{\log}$ at $U$ is the amalgamated sum
    	$$P\oplus_{\beta^{-1}_U(\cO_X(U)^*)}\cO_X(U)^*.$$
    	In particular,
    	the space of global sections of $P^{\log}$ is 
    	$P\oplus_{\beta^{-1}(A^*)}A^*$.
    \end{rk}
    
    \begin{definition}\label{der}
    	Let $f:X\to Y$ be a morphism of prelog schemes and $E$ an $\cO_X$-module. An $E$-\textit{valued derivation} of $f$ is a pair $(D,\delta)$ where $D:\cO_X\to E$ is a morphism of abelian sheaves and $\delta:M_X\to E$ is a morphism of sheaves of monoids such that the following conditions are satisfied:
    	\\ 
    	1) $D(\alpha_X(m))=\alpha_X(m)\delta(m)$ for all local sections $m$ of $M_X$;
    	\\
    	2) $\delta(f^\flat(n))=0$ for all local sections $n$ of $M_Y$;
    	\\
    	3) $D(ab)=aD(b)+bD(a)$ for all local sections $a,b$ of $\cO_X$;
    	\\
    	4) $D(f^\sharp(c))=0$ for all local sections $c$ of $f^{-1}\cO_Y$.
    \end{definition}
    
    We denote by $\Der_{X/Y}(E)$ the set of all $E$-valued derivations.
    As the analogue of differentials of schemes,
    $\Der_{X/Y}(E)$ is represented by the module of \textit{log differentials},
    given in the following result.
    
    \begin{thm}[\cite{ogus}-IV-1.2.4]\label{logder}
    	Let $f:X\to Y$ be a morphism of prelog schemes. Then the functor $E\mapsto\Der_{X/Y}(E)$ is representable by an $\cO_X$-module $\Omega_{X/Y}^1$ endowed with a universal derivation $d\in\Der_{X/Y}(\Omega_{X/Y}^1)$.
    \end{thm}
    
    For a morphism of log schemes $X\to Y$,  we denote by $\Omega_{\underline{X}/\underline{Y}}^1$ the sheaf of relative differentials of the underlying schemes. 
    Here are two concrete constructions of $\Omega_{X/Y}^1$ :
    \\ \hspace*{\fill} \\
    1) \cite{ogus}-IV-1.2.4: $\Omega_{X/Y}^1\simeq(\Omega_{\underline{X}/\underline{Y}}^1\oplus(\cO_X\otimes M_X^{gp}))/R$ where $R$ is the $\cO_X$-submodule generated by sections of the form
    $$(d\alpha_X(m),-\alpha_X(m)\otimes m),\quad \forall m\in M_X,$$
    $$(0,1\otimes f^\flat(n)),\quad \forall n\in f^{-1}M_Y.$$
    And the universal derivation is the evident maps
    $$d:\cO_X\to \Omega_{X/Y}^1,\quad d:M_X\to \Omega_{X/Y}^1$$
    where the latter $d$, sometimes denoted by $d\log$,
    is given by condition (1) in
    Definition \ref{der}.
    With this notation,
    we see that
    $d\alpha_X(m)=\alpha_X(m)dm$.
    For this reason we often write 
    \begin{equation}\label{logdiff}
    	1\otimes m=dm=\frac{d\alpha_X(m)}{\alpha_X(m)}
    \end{equation}
    in concrete examples;
    this also explains why 
    $d:M_X\to \Omega_{X/Y}^1$
    is written as $d\log$ in some other references.
    \\
    2) \cite{ogus}-IV-1.2.11: $\Omega_{X/Y}^1\simeq(\cO_X\otimes M_X^{gp})/(R_1+R_2)$ 
    induced by $dm\mapsto 1\otimes m$
    via the identification of $ \Omega_{X/Y}^1$ in (1),
    for $m\in M_X$,
    where
    \\
    $\bullet$ $R_1\subset \cO_X\otimes M_X^{gp}$ is the subsheaf of sections locally fo the form
    $$\sum_i\alpha_X(m_i)\otimes m_i-\sum_i\alpha_X(m_i')\otimes m_i',$$
    for $\sum_i\alpha_X(m_i)=\sum_i\alpha_X(m_i')$;
    \\
    $\bullet$ $R_2$ is the image of $\cO_X\otimes f^{-1}M_Y^{gp}\to\cO_X\otimes M_X^{gp}$.
    \\ \hspace*{\fill} 
    
    We briefly explain the above discussion in the language of log rings,
    which will be used in later construction.
    Let $\theta:(\alpha:A\to P)\to (\beta:B\to Q)$ be a morphism of log rings.
    Then $\theta$ gives rise to a morphism of affine log schemes
    $X\to Y$
    where $X=\Spec(Q\to B)$, $Y=\Spec(P\to A)$.
    The module of global sections of the $\cO_X$-module $\Omega_{X/Y}^1$ is isomorphic to
    $$(\Omega_{B/A}^1\oplus(B\otimes Q^{\gp}/P^{\gp}))/R.$$
    Here $R$ is the submodule of $(\Omega_{B/A}^1\oplus(B\otimes Q^{\gp}/P^{\gp}))$
    generated by elements of the form
    $$(d\beta(q),-\beta(q)\otimes\bar{\pi}(q))\quad\text{for}\ q\in Q$$
    where
    $\bar{\pi}:Q\to Q^{\gp}/P^{\gp}$
    is the canonical map induced by $\pi:Q\to Q^{\gp}$.
    Or alternatively, it is isomorphic to 
    $$(B\otimes Q^{\gp}/P^{\gp})/R'$$
    where $R'$ is the submodule generated by elements of the form
    $$\sum_i \beta(q_i)\otimes q_i-\sum_i \beta(q_i')\otimes q_i'\quad\text{for} \ \sum_i \beta(q_i)=\sum_i \beta(q_i').$$
     
     It is time to introduce the notion of smoothness for log schemes.
     Recall that a morphism of schemes is smooth if any infinitesimal thickening over it can be lifted.
     This idea is generalized to the category of log schemes as follows.
     
     \begin{definition}
     	A morphism $f:X\to Y$ of log schemes is \textit{strict} if the induced morphism of sheaves of monoids $f^*M_Y\to M_X$ is an isomorphism.
     	A \textit{log thickening} is a strict closed immersion $i:T'\to T$ of log schemes such that the square of the ideal sheaf $\cI$ of $T'$ in $T$ vanishes, and the subgroup $1+\cI$ of $\cO_{T}^*\simeq M_{T}^*$ operates freely on $M_T$. 
     	A \textit{log thickening over} $f$ is a commutative diagram
     	$$
     	\begin{tikzcd}\label{thi}
     		T' \arrow[d, "g"'] \arrow[r, "i"] & T \arrow[d, "h"] \\
     		X \arrow[r, "f"]                  & Y               
     	\end{tikzcd}
     	$$
     	where $i$ is a log thickening. 
     	A \textit{deformation} of $g$ to $T$ is an element of 
     	$$\operatorname{Def}_f(g,T):=\{\tilde{g}:T\to X:\tilde{g}\circ i=g,f\circ\tilde{g}=h\}.$$
        The morphism $f:X\to Y$ is \textit{formally smooth} (resp. \textit{unramified}, resp. \textit{\'{e}tale}) if for every log thickening $T'\to T$ over $f$, locally on $T$ there exists at least one (resp. at most one, resp. exactly one) deformation $\tilde{g}$ of $g$ to $T$.
        A morphism $f$ is \textit{smooth}
        (resp. \textit{\'{e}tale})
        if it is formally smooth 
        (resp. {\'{e}tale})
        and in addition $M_X$ and $M_Y$ are coherent and
        $\ul{f}$
        is locally of finite presentation.
    \end{definition}
    
    As a generalization of smoothness (resp. unramification, étaleness)
    for schemes,
    these properties are excepted to be compatible with the classical results.
    And fortunately we have the following result.
    (u-integral is a weaker condition compared to integral;
    a monoid $M$ is said to be \textit{u-integral} if 
    $m\in M$, $u'\in M^*$ and $m+u'=m$ implies that
    $u'=0$.)
    \begin{prop}[\cite{ogus}-IV-3.1.6]
    	\label{ogusetale}
    	Let $f:X\to Y$ be a strict morphism of log schemes.
    	If the morphism of underlying schemes 
    	$\ul{f}:\ul{X}\to\ul{Y}$
    	is formally smooth
    	(resp. étale, unramified),
    	then the same is true of $f$.
    	The converse holds if the log structure on $Y$ is u-integral.
    \end{prop}
    
    \begin{rk}\label{etaleproperty}
    	Some properties of
    	smooth (resp. unramified, étale)
    	morphism of schemes
    	also hold in log framework.
    	We list some of them which will be used later:
    	\\
    	1) (\cite{ogus}-IV-2.3.1) Let  $f:X\to Y$ and $g:Y\to Z$ be morphisms of log schemes.
    	Then there is an exact sequence of sheaves of $\cO_X$-modules
    	$$f^*\Omega_{Y/Z}^1
    	\to\Omega_{X/Z}^1\to\Omega_{X/Y}^1\to0.$$
    	2) (\cite{ogus}-IV-3.1.3) A morphism $f:X\to Y$ of log schemes is formally unramified if 
    	$\Omega_{X/Y}^1=0$.
    	The converse holds if $X$ and $Y$ are coherent.
    \end{rk}

    \begin{rk}
    	Note that whenever
    	we write $X_{\et}$
    	for a log scheme $X$,
    	it denotes the small étale site of the underlying scheme of $X$,
    	rather than the category of étale morphisms to the log scheme itself.
    \end{rk}
    
    \subsection{Vertex algebras}
	We recall the basic definitions of vertex algebras
	and their representations. 
	For more details on vertex algebras, 
	see \cite{frenkel-benzvi}.
	\begin{definition}
		A \textit{vertex algebra} is a vector space $V$ over a field of characteristic $0$, equipped with the \textit{vacuum vector} $\0$, 
		the \textit{translation operator} $T\in\End V$, and the \textit{vertex operation} $Y(-,z):V\to\End V[[z^{\pm1}]]$ 
		which  linearly takes each $A\in V$ to a \textit{field} $$Y(A,z)=\sum_{n\in\bZ}A_nz^{-n-1}$$ 
		such that the following axioms are satisfied:
		\\
		1) \textit{Vacuum axiom}: $Y(\0,z)=\id_V$. Furthermore, $A_n\0=0$ for $n\geq0$ and $A_{-1}\0=A$;
		\\
		2) \textit{Translation axiom}: for any $A\in V$, $[T,Y(A,z)]=\partial_z Y(A,z)$;
		\\
		3) \textit{Locality axiom}: for any $A,B\in V$, there exists some $N\in\bZ_{\geq0}$ such that $$(z-w)^N[Y(A,z),Y(B,w)]=0.$$
	\end{definition}
	
	A vertex algebra $V$ is called $\bZ$-graded if $V$ is a $\bZ$-graded vector space, 
	$\0$ is a vector of degree $0$, 
	$T$ is a linear operator of degree $1$, 
	and for $A\in V_m$, $\deg A_n=-n-1+m$. 
	In many texts,
	the locality axiom is called \textit{weak commutativity}.
	
	\begin{definition}
		A $\bZ$-graded vertex algebra $V$ is called \textit{conformal of central charge} $c\in\bC$, if there is a given non-zero \textit{conformal vector} $\omega\in V_2$ such that the Fourier coefficients $L_n$,
		$n\in\bZ$ of the corresponding vertex operator $Y(\omega,z)=\sum_{n\in\bZ}L_nz^{-n-2}$ satisfy the following conditions:
		\\
		1) $[L_n,L_m]=(n-m)L_{n+m}+\frac{n^3-n}{12}\delta_{n,-m}c\id_V$;
		\\
		2) $L_{-1}=T$, $L_0|_{V_n}=n\id_{V_n}$.
	\end{definition}
	
	\begin{rk}\label{jacobi}
		As an analogue of Lie algebra,
		a vertex algebra carries the
		\textit{Jacobi identity}:
		$$
		\sum_{i\geq0}(-1)^i{{l}\choose{i}}A_{m+l-i}B_{n+i}
		-
		(-1)^l\sum_{i\geq0}(-1)^i{{l}\choose{i}}B_{n+l-i}A_{m+i}
		=
		\sum_{i\geq0}{{m}\choose{i}}(A_{l+i}B)_{m+n-i}.$$
		As shown in \cite{lepowsky2004introduction}-3.4.1\&3.5.1
		that
		the Jacobi identity is equivalent to the weak commutativity together with the translation axiom.
		Here are two useful formulas \cite{lepowsky2004introduction}-(3.1.9)\&(3.1.12) derived from the Jacobi identity,
		which will be used in the subsequent proofs:
		$$
		[A_m,B_n]=\sum_{i\geq0}{{m}\choose{i}}(A_iB)_{m+n-i};$$
		$$(A_mB)_n=
		\sum_{i\geq0}(-1)^i{{m}\choose{i}}
		(A_{m-i}B_{n+i}-(-1)^mB_{m+n-i}A_i)$$
 		for $A,B\in V$, $m,n\in\bZ$.
	\end{rk}
	
	\begin{definition}[\cite{frenkel-benzvi}-6.4.1]
		\label{primary}
		A vector $A$ in a conformal vertex algebra $V$ is called  \textit{primary
		of conformal dimension} $\Delta\in\bZ_{>0}$ if
		$$L_0A=\Delta A,\quad L_nA=0\quad\text{for}\ n>0.$$
	\end{definition}

    \begin{rk}
    	In \cite{frenkel-benzvi}-6.4,
    	a primary vector is defined in a so-called \textit{quasi-conformal} vertex algebra
    	(a weaker notion than that of a conformal vertex algebra),
    	but we do not  necessarily need it.
    \end{rk}
	
	\begin{definition}
		Let $(V,\0,T,Y)$ be a vertex algebra. A vector space $M$ is called a $V$-\textit{module} if it is equipped with an operation $Y_M:V\to\End M[[z^{\pm1}]]$ which assigns to each $A\in V$ a field $Y_M(A,z)=\sum_{n\in\bZ}A_n^Mz^{-n-1}$ subject to the following axioms:
		\\ 
		1) $Y_M(\0,z)=\id_M$, and $Y_M(A,z)m\in M((z))$ for all $A\in V$, $m\in M$;
		\\
		2) for all $A,B\in V$, there exists some $N\in\bZ_{\geq0}$ such that $$(z-w)^N[Y_M(A,z),Y_M(B,w)]=0;$$
		3) for all $A\in V$, there exists some $l\in\bZ_{\geq0}$ such that for any $B\in V$,
		$$(z+w)^lY_M(Y(A,z)B,w)=(z+w)^lY_M(A,z+w)Y_M(B,w).$$
	\end{definition}
		If $V$ is  $\bZ$-\textit{graded}, then $M$ is called $\bZ$-graded if $M$ is a $\bC$-graded vector space and for $A\in V_m$, $\deg A_n^M=-n-1+m$. 
		If $V$ is conformal with conformal vector $\omega$, then $M$ is called a \textit{conformal} $V$-\textit{module} if addition that the Fourier coefficient $L_0^M$ of the field $Y_M(\omega,z)=\sum_{n\in\bZ}L_n^Mz^{-n-2}$ acts semi-simply on $M$.

    For a vertex  algebra $(V,\0,T,Y)$,
    let $\partial=T\otimes\id+\id\otimes\partial_t$
    be a linear operator on $V\otimes\bC[t^{\pm1}]$. 
    The completion of $U'(V):=V\otimes\bC[t^{\pm1}]/\im\partial$ with respect to the $(t)$-adic topology on $\bC[t^{\pm1}]$ is denoted by $U(V)$.
    We denote by $A_{[n]}$ the projection of $A\otimes t^k\in V\otimes\bC[t^{\pm1}]$
    in $U'(V)$,
    and then $U'(V)$ is spanned by $A_{[n]}$, $A\in V$, $n\in\bZ$, with relation
    $(TA)_{[n]}=-nA_{[n-1]}$.
    The bilinear map 
    $[\ ,\ ]:U'(V)^{\otimes 2}\to U'(V)$
    defined by
    $$[A_{[m]},B_{[k]}]=
    \sum_{n\geq0}{{m}\choose{n}}(A_nB)_{[m+k-n]}$$
    gives rise to a Lie bracket on $U'(V)$ (\cite{frenkel-benzvi}-4.1.2).
    This Lie bracket is continuously extended to $U(V)$.
    If we give a $\bZ$-gradation on $U(V)$ by setting
    $\deg A_{[n]}=\deg A-n-1$,
    then
    there is a natural Lie algebra homomorphism (\cite{frenkel-benzvi}-4.1.2)
    \begin{equation*}
    	\begin{aligned}
    		U(V)&\to \End V
    		\\
    		\sum_{n\geq N}f_n A_{[n]}&\mapsto
    		\Res_{z=0}Y(A,z)f(z)dz
    	\end{aligned}
    \end{equation*}
    preserving the $\bZ$-gradation,
    where $f(z)=\sum_{n\geq N}f_nz^n\in\bC((z))$.
    
    The above construction of $U'(V)$ and $U(V)$ will play an important role in the proof of Proposition \ref{lieD},
    together with the following example,
    which suggests that vertex algebra is also analogous to the notion of commutative associative algebra.
    
	\begin{ex}[\cite{frenkel-benzvi}-1.4]\label{comm}
		Let $V$ be a commutative $\bC$-algebra with a unit $\0$ and finite dimensional homogeneous components,
		with a derivation $T$ of degree $1$.
		Then $(V,\0,T,Y)$ will be a vertex algebra if we define
		$$Y(A,z):=\sum_{n\geq0}\frac{z^n}{n!}(T^nA)=e^{zT}A.$$
		Firstly
		the series $Y(A,z)$ is clearly a field as 
		$Y(A,z)B\in V[[z]]$.
		Let us check the remaining axioms.
		\\
		1) vacuum axiom:
		 We have $T\0=0$ by the definition of derivation,
		 and thus
		$Y(\0,z)=e^{zT}\0=\id_V$.
		From the definition of $Y(A,z)$ observe that
		$A_n=\frac{1}{(-n-1)!}T^{-n-1}A$
		for $n\leq-1$,
		and $A_n=0$ for $n\geq0$.
		Then clearly $A_n\0=0$ for $n\geq0$,
		and $A_{-1}\0=A\0=A$;
		\\
		2) translation axiom:
		One has
		$$[T,Y(A,z)]=\sum_{n\leq-1}\frac{z^{-n-1}}{(-n-1)!}[T,T^{-n-1}A],\quad
		\partial_zY(A,z)=
		\sum_{n\leq-1}\frac{z^{-n-1}}{(-n-1)!}T^{-n}A.$$
		For $B\in V$, it follows that
		\begin{equation*}
			\begin{aligned}
				[T,T^{-n-1}A]B&=T((T^{n-1}A)B)-(T^{n-1}A)(TB)
				\\
				&=(T^nA)B+(T^{n-1}A)(TB)-(T^{n-1}A)(TB)=(T^nA)B
			\end{aligned}
		\end{equation*}
		for $n\geq1$,
		and thus $[T,Y(A,z)]=\partial_zY(A,z)$.
		\\
		3) locality axiom: Since $Y(A,z)$ and $Y(B,w)$ are regular in $z$ and $w$ respectively from the definition,
		immediately it holds that
		$[Y(A,z),Y(B,w)]=0$.
		\\
		Moreover,
		for $A\in V_m$,
		it is easy to see that
		$\deg A_n=\deg T^{-n-1}A=-n-1+m$.
		Hence if the degree of the unit $\0$ is $0$,
		the vertex algebra $V$ will be $\bZ$-graded. 
	\end{ex}
    
    We end this section with the tensor product of vertex algebras.
    Because the concrete construction will be needed later,
    we provide a proof.
    
    \begin{prop}[\textbf{Tensor products of vertex algebras} \cite{lepowsky2004introduction}-3.12.5
    ]\label{tensor}
    	Let $(V_i,\0^{(i)},T^{(i)},Y^{(i)})$, $i=1,...,r$ be vertex algebras.
    	Denote by
    	$$V=V_1\otimes\cdots\otimes V_r$$
    	and define
    	$$Y(A^{(1)}\otimes\cdots\otimes A^{(r)},z)=
    	Y^{(1)}(A^{(1)},z)\otimes\cdots\otimes
    	Y^{(r)}(A^{(r)},z)$$
    	$$\0=\0^{(1)}\otimes\cdots\otimes \0^{(r)}$$
    	$$T=\sum_{i=1}^r
    	\id_{V_1}\otimes\cdots\otimes
    	T^{(i)}\otimes\cdots\otimes
    	\id_{V_r}$$
    	for $A^{(i)}\in V_i$, $i=1,...,r$.
    	Then $(V,\0,Y,T)$ carries a structure of a vertex algebra.
    \end{prop}
    
    \begin{proof}
    	Writing the vertex operator as 
    	$$Y(A^{(1)}\otimes\cdots\otimes A^{(r)},z)=
    	\sum_{n\in\bZ}
    	\left(
    	\sum_{n_1+...+n_r=n}
    	A^{(1)}_{n_1}\otimes\cdots\otimes A^{(r)}_{n_r}
    	\right)
    	z^{-n-r}$$
    	we see that
    	$$(A^{(1)}\otimes\cdots\otimes A^{(r)})_{-1}\0=
    	\left(
    	\sum_{n_1+...+n_r=-r}
    	A^{(1)}_{n_1}\otimes\cdots\otimes A^{(r)}_{n_r}
    	\right)
    	(\0^{(1)}\otimes\cdots\otimes \0^{(r)}).
    	$$
    	Since $A^{(i)}_{n_i}\0^{(i)}=0$ for $n_i\geq0$,
    	the nonzero term of 
    	$(A^{(1)}\otimes\cdots\otimes A^{(r)})_{-1}\0$
    	in the summation
    	is of the case $n_1=...=n_r=-1$,
    	which implies that
    	$$(A^{(1)}\otimes\cdots\otimes A^{(r)})_{-1}\0
    	=
    	A^{(1)}_{-1}\0^{(1)}\otimes\cdots\otimes A^{(r)}_{-1}\0^{(r)}=
    	A^{(1)}\otimes\cdots\otimes A^{(r)}.$$
    	This shows the vacuum axiom.
    	The translation axiom is satisfied from
    	\begin{equation*}
    		\begin{aligned}
    			[T,Y(A^{(1)}\otimes\cdots\otimes A^{(r)},z)]&=
    			\left[\sum_{i=1}^r
    			\id_{V_1}\otimes\cdots\otimes
    			T^{(i)}\otimes\cdots\otimes
    			\id_{V_r},
    			Y^{(1)}(A^{(1)},z)\otimes\cdots\otimes
    			Y^{(r)}(A^{(r)},z)
    			\right]
    			\\
    			&=
    			\sum_{i=1}^r
    			Y^{(1)}(A^{(1)},z)\otimes\cdots\otimes
    			[T^{(i)},Y^{(i)}(A^{(i)},z)]
    			\otimes\cdots\otimes
    			Y^{(r)}(A^{(r)},z)
    			\\
    			&=
    			\sum_{i=1}^r
    			Y^{(1)}(A^{(1)},z)\otimes\cdots\otimes
    			\partial_zY^{(i)}(A^{(i)},z)
    			\otimes\cdots\otimes
    			Y^{(r)}(A^{(r)},z)
    			\\
    			&=\partial_zY(A^{(1)}\otimes\cdots\otimes A^{(r)},z).
    		\end{aligned}
    	\end{equation*}
        Because the Lie bracket in each $\End V_i$ is the commutator,
        we have that
        \begin{equation*}
        	\begin{aligned}
        		[Y(A^{(1)}\otimes\cdots\otimes A^{(r)},z)&,
        		Y(B^{(1)}\otimes\cdots\otimes B^{(r)},w)]\\
        		&=[Y^{(1)}(A^{(1)},z), Y^{(1)}(B^{(1)},z)]\otimes\cdots\otimes
        		[Y^{(r)}(A^{(r)},z),Y^{(r)}(B^{(r)},z)].
        	\end{aligned}
        \end{equation*}
        By assumption,
        there is a positive integer $N_i$ for each $i=1,...,r$ such that 
        $$(z-w)^{N_i}[Y^{(i)}(A^{(i)},z), Y^{(i)}(B^{(i)},z)]=0.$$
        Then setting $N=N_1+\cdots+N_r$,
        we have
        $$(z-w)^N [Y(A^{(1)}\otimes\cdots\otimes A^{(r)},z),
        Y(B^{(1)}\otimes\cdots\otimes B^{(r)},w)]=0.$$
        This leads to the locality axiom,
        and therefore we complete the proof.
    \end{proof}
	
	\section{Coordinate system}
	
	In this section, 
	we describe the coordinate system, 
	which behaves subtly near nodes in the usual sense.
	The main idea of this section is adapted from \cite{carnahan2015equivariant},
	with some modifications.
	\par 
	Let $S$ be an fs log scheme, and denote by $\LS_S$ the category of fs log schemes over $S$. Let $X\in\LS_S$ be  a log curve over $S$ in the sense of \cite{fkato}, i.e.,
	the morphism of fs log schemes $X\to S$ is smooth and integral, with reduced and connected curves as geometric fibers.
	From \cite{ogus}-III-2.1.2,
	we know that the tensor product of log schemes is well-defined and the functor
	$X\to\ul{X}$ taking a log scheme to its underlying scheme commutes with the fiber product.
	Thus for $T\in\LS_S$, 
	we denote by $X_T$ the $T$-log scheme arising from
	the fiber product of $X$ and $T$ in the category $\LS_S$.
    \par 
    The (formal) \textit{disc} $D$ is defined to be the formal scheme $\Spf\bZ[[t]]$, endowed with the log structure induced by the chart $\bN\to\bZ[[t]]$, $n\mapsto t^n$. 
    It is the direct limit of the jet schemes $D_n=\Spec (\bN\to\bZ[t]/(t^{n+1}))$. Let $D^\times=\Spec\bZ((t))$ be the \textit{punctured disc}, with the trivial log structure. 
    The usual abstract \textit{disc} is defined by $D^\circ=\Spec\bZ[[t]]$, with the same log structure as that on $D$.
    In fact, 
   the log structure on $D^\circ$ is the direct image (recall Definition \ref{direct}) 
   of the trivial log structure on $D^\times$.
   Through the natural embedding
   $D^\times\hookrightarrow
   D^\circ$
   it holds that
   $$
   \bZ((t))^*\oplus_{\bZ((t))}\bZ[[t]]\simeq
   \bZ((t))^*\cap\bZ[[t]]=t^\bN\bZ[[t]]^*\simeq
   \bN\oplus\bZ[[t]]^*$$
   where the last isomorphism
   of monoids gives rise to an isomorphism of log structures on $D^\circ$.
    
    \begin{def-lem}
    	The presheaf $\Autu\cO$ of groups on $\LS_S$ is  defined by 
    	$$\Autu\cO(T)=\Aut_T(D_T)$$
    	as the space of continuous $T$-automorphisms of $D_T$. Moreover the subpresheaf $\Autu^0\cO$ of $\Autu\cO$ whose value $\Autu^0\cO(T)$ at $T$ is defined to be the space of $T$-automorphisms of $D_T$ preserving the base points (this is equivalent to saying that which preserves the point generated by the topological generator $t$). 
    	We denote their sheafifications by the same symbols.
    \end{def-lem}

    \begin{proof}
    	We need to prove that 
    	$\Autu\cO$ and 
    	$\Autu^0\cO$
    	are functors.
    	Let $U\to V$ be a morphism in $\LS_S$. Then $D_U$ is the pull-back of $U\to V\leftarrow D_V$, i.e., 
    	the base change of $D_V$ along $U\to V$. 
    	This implies that a $V$-automorphism of $D_V$ can be uniquely lifted to a $U$-automorphism of $D_U$,
    	which indicates that we have a map
    	$\Autu\cO(V)\to
    	\Autu\cO(U)$.
    	Locally we suppose that $V=\Spec R$ (resp. $U=\Spec S$) is endowed with the log structure modeled by the chart $Q_V$ (resp. $Q_U$). The morphism $U\to V$ is equivalent to
    	$$
    	\begin{tikzcd}
    		R \arrow[r]             & S             \\
    		Q_V \arrow[u] \arrow[r] & Q_U \arrow[u]
    	\end{tikzcd}
        $$
        For $\vp\in\Autu^0\cO(V)$, we need to check whether the extension $\vp_U$ of $\vp$ from $V$ to $U$ preserves the base point of $D_U$ generated by $t$. 
        For $f(t)\in tS[[t]]$,
        let $f(t)=\sum_{i>0}s_it^i$ for $s_i\in S$, 
        then
        $\vp^\sharp_Uf(t)=\sum s_i\vp^\sharp(t^i)$.
        Since $\vp^\sharp(t^i)\in tR[[t]]$, we have that $\vp^\sharp_Uf(t)\in tS[[t]]$. Namely, $\vp_U\in\Autu^0\cO(U)$.
        Hence we obtain a map
        $\Autu^0\cO(V)\to\Autu^0\cO(U)$
        associated to $U\to V$.
    \end{proof}
    
    A $D_T$-\textit{point} on $X_T$ over $T$ is defined to be a morphism of $T$-log schemes $x:D_T\to X_T$.
    Each $D_T$-point $x$ uniquely induces a 
    $T$-log scheme morphism $x^\circ:D^\circ_T \to X$
    whose locally associated log ring map coincides with that of $x$.

    \begin{def-lem}
    	The functor $\Fra^\infty(T):X_\et\to \mathbf{Set}$ is defined as follows. 
    	For $U\in X_\et$,
    	the value of $\Fra_U^\infty$ at $T$ is defined to be 
    	$$
    	\begin{aligned}
    		\Fra&_U^\infty(T)=
    		\\
    		&\left\{z\in \Hom^{\fet}_T(D_T,U_T): 
    		\begin{tikzcd}
    			D^\circ_T \arrow[r, "z^\circ"] \arrow[rd, "\theta_z"', dashed] & U_T                             \\
    			& D^\circ_T \arrow[u, "x^\circ"']
    		\end{tikzcd}\ 
    		\text{for some}\ D_T-\text{point}\ x\ \text{and}\ \theta_z\in \Aut^0_T(D^\circ_T)\right\}.
    	\end{aligned}
    	$$ 
    	where $\Hom^{\fet}$ denotes formally \'{e}tale morphisms,
    	and $\Aut^0_T(D^\circ_T)$
    	is the collection of continuous $T$-automorphisms of $D^\circ_T$ preserving base points.
    \end{def-lem}

    \begin{proof}
    	We need to show that
    	$\Fra^\infty(T)$
    	is a functor.
    	Let $\ul{f}:U\to V$ be a morphism in $X_\et$, then $\ul{f}$ is \'{e}tale, and so is the extension $\ul{f}_T$.
    	Thus $f_T$ is also formally étale by Proposition \ref{ogusetale}
    	as the log structures on $U$ and $V$ are both derived from that on $T$.
    	Since (formally) \'{e}taleness is preserved under composition,
    	it follows that $f_T\circ z^\circ$ is formally \'{e}tale.
    	Next we show that if such $x^\circ$ exists,
    	then it is formally \'{e}tale.
    	By the following Theorem \ref{exactseq}
    	and Remark \ref{etaleproperty}-(2),
    	we have an isomorphism
    	$\theta^*_z\Omega^1_{x^\circ}\simeq\Omega^1_{z^\circ}=0$.
    	As $\theta_z$ is an isomorphism,
    	it follows that $\Omega_{x^\circ}^1=0$,
    	and thus $x^\circ$ is formally unramified.
    	Let $J'\to J$ be an arbitrary log thickening
    	over $x^\circ$, i.e.,
    	$$
    	\begin{tikzcd}
    		D^\circ_T \arrow[r, "x^\circ"]    & U_T              \\
    		J' \arrow[r, "i"] \arrow[u, "g'"] & J \arrow[u, "g"]
    	\end{tikzcd}
        $$
        If we denote by $J''$ the base change of $J'\to D^\circ_T$ via $\theta_z$,
        there will be a commutative diagram
        $$
        \begin{tikzcd}
        	D^\circ_T \arrow[r, "\theta_z"]                    & D^\circ_T \arrow[r, "x^\circ"]    & U_T              \\
        	J'' \arrow[r, "\tilde{\theta}_z"] \arrow[u, "g''"] & J' \arrow[r, "i"] \arrow[u, "g'"] & J \arrow[u, "g"]
        \end{tikzcd}
        $$
        In particular
        $\tilde{\theta}_z:J''\simeq J'$ since $\theta_z$ is an automorphism.
        As $z^\circ=x^\circ\circ\theta_z$ is formally \'{e}tale
        and 
        $i \circ\tilde{\theta}_z$ is a log thickening as well,
        there exists a unique $T$-morphism $\tilde{g}'':J\to D^\circ_T$ such that
        $g=x^\circ\circ\theta_z\circ\tilde{g}''$ and
        $g''=\tilde{g}''\circ i\circ\tilde{\theta}_z$.
        Let $\tilde{g}'=\theta_z\circ\tilde{g}''$.
        Then it follows immediately that
        $x^\circ\circ\tilde{g}'=x^\circ\circ \theta_z\circ\tilde{g}''=g$.
        Besides,
        we have
        $\tilde{g}'\circ i\circ\tilde{\theta}_z=\theta_z\circ\tilde{g}''\circ i\circ\tilde{\theta}_z=\theta_z\circ g''=g'\circ\tilde{\theta}_z$.
        Thus $\tilde{g}'\circ i=g'$ because $\tilde{\theta}_z$ is an automorphism.
        This proves that $x^\circ$ is formally smooth,
        and 
        therefore formally \'{e}tale.
    	Moreover, $f_T\circ x^\circ$ is formally \'{e}tale, and it is associated to a $D_T$-point on $V_T$. 
    	Hence we have proved that a morphism $\Fra_U^\infty(T)\to\Fra_V^\infty(T)$
    	derived from $U\to V$.
    	This implies that $\Fra^\infty(T)$ is a functor.
    \end{proof}

    \begin{rk}
    	The condition for
    	$z\in\Hom^{\fet}_T(D_T,U_T)$ 
    	to be an element in $\Fra^\infty(T)$
    	is equivalent to
    	\begin{equation*}
    		\begin{tikzcd}
    			D_T \arrow[r, "z"] \arrow[rd, "\theta_z"', dashed] & U_T                 \\
    			& D_T \arrow[u, "x"']
    		\end{tikzcd}
    	    \quad
    	    \text{for some}\ D_T-\text{point}\ x\ \text{and}\ \theta_z\in \Autu^0\cO(T).
    	\end{equation*}
    	However we use the expression as in the above Definition–Lemma, 
    	since this makes the proof easier to follow.
    \end{rk}
    
    \begin{thm}[\cite{ogus}-IV-3.2.3]\label{exactseq}
    	Let $X\to Y$ and $g:Y\to Z$ be morphisms of coherent log schemes,
    	and consider the exact sequence from Remark \ref{etaleproperty}-(1)
    	$$f^*\Omega_{Y/Z}^1
    	\stackrel{s}{\longrightarrow}
    	\Omega_{X/Z}^1
    	\stackrel{}{\longrightarrow}
    	\Omega_{X/Y}^1
    	\longrightarrow
    	0.$$
    	If $f$ is smooth, 
    	the map $s$ is injective and locally split.
    	In particular if $f$ is étale,
    	then $s$ is an isomorphism.
    \end{thm}

    Elements of $\Fra_U^\infty(T)$ and $\Autu^0\cO(T)$ are called $T$-\textit{valued coordinates on} $U$   and $T$-\textit{valued coordinate transformations} respectively. 
    We simply call the functors $\Fra^\infty$ and $\Autu^0\cO$ the \textit{space of coordinates} and \textit{space of coordinate transformations} respectively.
    
    \begin{rk}
    If we ignore the base change and take the base scheme $S$ to be the trivial log scheme $\Spec(0\to\bZ)$, 
    then locally a coordinate on
    $\Spec(Q_A\to A)$
    over $S$
    is given by the commutative diagram
    $$
    \begin{tikzcd}
    	Q_A \arrow[d] \arrow[r] & \mathbb{N} \arrow[d, "n\mapsto t^n"] \\
    	A \arrow[r]             & {\mathbb{Z}[[t]]}                   
    \end{tikzcd}
    $$
    whose corresponding log scheme morphism is formally \'{e}tale.
    \end{rk}
    
     \begin{rk}\label{pi}
    	For $U\in X_{\et}$
    	and $T\in\LS_S$, there is a natural map $\pi: \Fra_U^\infty(T)\to U_T(T)$, induced by composing with the base point. 
    	This is intuitively the projection from the space of coordinates to the underlying points.
    	More precisely, 
    	a base point $b:T\to D_T$ is locally given by
    	$$
    	\begin{tikzcd}
    		{R[[t]]} \arrow[r]              & R           &  & \alpha_T(q)t^n \arrow[r, maps to]             & \alpha_T(q)0^n          \\
    		Q\oplus \bN \arrow[r] \arrow[u] & Q \arrow[u] &  & {(q,n)} \arrow[u, maps to] \arrow[r, maps to] & 0^nq \arrow[u, maps to]
    	\end{tikzcd}
        $$
        where $T$ is locally of the form $\Spec(Q\to R)$,
        and $\pi$ maps a local coordinate $D_T \to U_T$ to the composition
        $T\stackrel{b}{\to} D_T\to U_T$.
    \end{rk}

    \begin{rk}
    	Let $y$ be a geometric point of the underlying scheme of $S$.
    	Then the geometric fiber  $X_y$ of $X\to S$ at $y$
    	is reduced and connected by definition.
    	Moreover,
    	it is shown (\cite{fkato}-1.3)
    	that
    	$\ul{X}_y$ has at most ordinary double points,
    	i.e.,
    	it is either of the form
    	$\Spec k(y)[t]$ or 
    	$\Spec k(y)[u,v]/(uv)$
    	étale locally.
    	If a point $x\in\ul{X}_y$ is smooth in the usual sense,
    	then there is a Zariski open subset 
    	$U\subset \ul{X}_y$
    	containing $x$
    	with an étale morphism
    	$f:U\to \bA^1_{k(y)}$.
    	In this case a local coordinate is classically understood as
    	$$\ul{D_S}\to U\to 
    	\bA^1_{k(y)}$$
    	where the base point generated by $t$ in $\ul{D_S}$ is mapped to $x$.
    	Thus our definition of coordinates coincides with the usual one in \cite{frenkel-benzvi}.
    	On the other hand, 
    	if $x$ is the node on a geometric fiber of $X\to S$,
    	then we have a commutative diagram
    	$$
    	\begin{tikzcd}
    		D_S \arrow[d] \arrow[rd] &                               \\
    		U \arrow[r, "f"]         & {\Spec\frac{k(y)[u,v]}{(uv)}}
    	\end{tikzcd}
        $$
        with the two arrows forming the top side of the triangle formally étale.
        Intuitively,
        the classical coordinate is motivated by considering the affine line as a standard model,
        while in our setting we also allow 
        normal crossings,
        and together these constitute a new standard model for coordinates.
        Details about coordinates near nodes will be discussed in Section \ref{casei}.
    \end{rk}
    
    \begin{definition}
    	For a $D_T$-point $x$ on $X_T$, the \textit{space of local coordinates near} $x$ is defined to be  
        $$\Aut_x=\left\{z\in \Hom^{\fet}_T(D_T,X_T): 
        \begin{tikzcd}
        	D^\circ_T \arrow[r, "z^\circ"] \arrow[rd, "\theta_z"', dashed] & X_T                             \\
        	& D^\circ_T \arrow[u, "x^\circ"']
        \end{tikzcd}\ 
        \text{for some}\ 
        \theta_z\in \Aut^0_T(D^\circ_T)
        \right\}.$$
    \end{definition}
	
	From the definition of $\Aut_x(T)$, it is easy to see that it admits an $\Autu^0\cO(T)$-action by automorphisms of $D_T$,
	$$
	\begin{tikzcd}
		D_T \arrow[rd, "\text{formally \'{e}tale}"'] \arrow[r, "\sim"] & D_T \arrow[d, "\text{formally \'{e}tale}"] \\
		& X_T                                       
	\end{tikzcd}
	$$
	Each $\rho\in\Autu^0\cO(T)$ induces an $T$-automorphism of $D^\circ_T$ preserving the base points, 
	which we denote by the same symbol $\rho$ if there is no confusion.
	If $z$ is a given coordinate near $x$, then $z^\circ\circ \rho$ is a formally \'{e}tale morphism $D^\circ_T\to X_T$.
	Furthermore, it is clear that 
	$x^\circ\circ\theta_z\circ\rho=z^\circ\circ\rho$, and thus $\rho\cdot z$ is a coordinate near $x$.
	Therefore we have the following result.
	
	\begin{prop}\label{thm:coor}
		The action of $\Autu^0\cO(T)$ on $\Aut_x$ is simply transitive.
	\end{prop}

    In general, a morphism $\ul{D^\circ}\to \ul{X}$ cannot be \'{e}tale, but we can still define $\Fra_{D^\circ}^\infty$ in the same way,
    whose value at $T\in\LS_S$ is  $\Hom^{\fet}_T(D_T,D^\circ_T)$. 
    When $T=\Spec(0\to\bZ)$,
    it is easy to see that the point $(p)\in D_T$, 
    where $p$ is a prime number, must be mapped to $(p)\in D^\circ_T$ under a formally \'{e}tale morphism,
    and thus we have an isomorphism
    $\Hom^{\fet}_T(D_T,D^\circ_T)\simeq 
    \Hom^{\fet}_T(D_T^\circ,D^\circ_T)$.
    This intuitively inspires us to construct a torsor over $X$, 
    via the following result.
    
    \begin{prop}\label{p2}
    	There is an isomorphism of functors $\Fra^\infty_{D^\circ}\simeq\Autu^0\cO$.
    \end{prop}

    \begin{proof}
    	Since any isomorphism is \'{e}tale, 
    	it suffices to show that any
    	$\alpha\in\Fra^\infty_{D^\circ}(\bZ)\simeq\Hom_\bZ^\fet(D^\circ,D^\circ)$ is an automorphism of $D^\circ$ preserving the base point $(t)$.
    	Let $\alpha^\sharp(t)=\sum_{i\geq0}a_it^i$. 
    	Since $\alpha^\sharp(t\pm1)=\alpha^\sharp(t)\pm1$ are units and the constant term of every unit in $\bZ[[t]]$ is $\pm1$, it follows that $a_0=0$.
    	Namely, $\alpha^\sharp(t)\in t\bZ[[t]]$. 
    	Moreover, $a_1=\pm1$, otherwise suppose that
    	there is some prime $p$ dividing $a_1$.
    	Now consider the base change $\alpha^\sharp_p$ of $\alpha^\sharp$
    	by the modulo map $\bZ\to\bF_p$,
    	then 
    	$\alpha^\sharp_p(t)\in t^2\bF_p[[t]]$,
    	and
    	$d\alpha^\sharp_p(t)=2a_2dt+O(t^2)dt\mod p$.
    	Then it is easy to see that 
    	$\alpha^\sharp_p$ is not formally unramified as $2a_2$ is always an even number and thus $dt\not=0$ in 
    	$\Omega^1_{\alpha_p}$.
    	Therefore $\alpha^\sharp$ will also not be unramified if
    	$a_1\not=\pm1$,
    	and hence
    	$\alpha^\sharp(t)=\pm t+a_2t^2+a_3t^3+...$
    	Thus $\alpha$ preserves the base point $(t)$.
    	Next,
    	we show that $\alpha$ is an automorphism. For $f=b_0+b_1t+...\in\bZ[[t]]$ such that $\alpha^\sharp(f)=b_0+b_1\alpha^\sharp(t)+b_2\alpha^\sharp(t)^2+...=0$, obviously $b_0=0$. Then $\alpha^\sharp(t)(b_1+b_2\alpha^\sharp(t)+...)=0$. Since $\bZ[[t]]$ is an integral domain, it derives that $b_1+b_2\alpha^\sharp(t)+...=0$; therefore $b_1=0$. 
    	Hence $\alpha^\sharp(t)(b_2+b_3\alpha^\sharp(t)+...)=0$. Inductively all $b_i=0$, so $f=0$, and thus $\alpha^\sharp$ is an injective endomorphism.
    	Moreover, since $\bZ[[t]]$ is a noetherian ring, it follows that $\alpha^\sharp$ is an automorphism
    	as injective endomorphisms of noetherian rings are automorphisms. 
    	Now we fit this into the log framework.
    	For a natural number $m$, we have that $\alpha^\sharp(t)^m=\sum a_{i_1}\cdots a_{i_m}t^{i_1+...+i_m}=t^mh(t)$
    	for some unique $h(t)\in\bZ[[t]]^*$.
    	The associated monoid isomorphism
    	$\bN\oplus\bZ[[t]]^*\to \bN\oplus\bZ[[t]]^*$
    	is given by
    	$\alpha^\flat(m,f(t))=(m,f(\alpha^\sharp(t))h(t))$ with $h(t)=t^{-m}\alpha^\sharp(t)$.
    	This completes the proof.
    \end{proof}

    For a $D_T$-point $x:D_T\to X_T$ over $T$, we sometimes denote $x$ by $D_x$,
    and let $\Fra^\infty_{D_x}$ denote the space $\Fra_D^\infty(T)$ with respect to $x$.
    More concretely,
    if we 
    suppose $x$ maps $D_T$
    into an open subset $U_T$ of $X_T$,
    then the map 
    $\Fra_U^\infty(T)\to U_T(T)$
    from Remark \ref{pi} is decomposed into
    $$
    \begin{tikzcd}
    	\Fra_U^\infty(T) \arrow[r] \arrow[d] & U_T(T)           \\
    	\Fra_D^\infty(T) \arrow[r]           & D_T(T) \arrow[u]
    \end{tikzcd}
    $$
    such that the right-hand side vertical morphism is induced by $x$.
    From the above proposition, we know that $\Fra^\infty_{D_x}$ admits a simply transitive $\Autu^0\cO(T)$-action as well. 
    Furthermore
    $\Fra^\infty_{D_x}\simeq \Fra^\infty_D(T)$ 
    carries an affine scheme structure over $T$. Explicitly,
    it follows from the proof of Proposition \ref{p2} that
    \begin{equation*}
    	\begin{aligned}
    		\Fra^\infty_{D_x}\simeq\Fra^\infty_D(T)&\simeq\underline{\Spec}_T\cO_T[a_1,a_1^{-1},a_2,a_3,...]\\
    		&=\dlim_{N}\underline{\Spec}_T\cO_T[a_1,a_1^{-1},a_2,...,a_N]
    	\end{aligned}
    \end{equation*}
    is a proalgebraic group over $T$,
    where $a_i$'s can be considered as coefficients of the coordinate transformation series in
    $\Autu^0\cO(T)$.
    This inspires us to consider $\Aut_x$ as a trivial torsor $\Autu^0\cO(T)\times D_x\simeq \Fra^\infty_{D_x}\times D_x$ over $D_x$.

    \section{Vertex algebra bundles and conformal blocks}\label{vacb}
    
    \subsection{Sheaf of vertex algebras}
    
    Let $V$ be a conformal vertex algebra bounded from below, 
    and let $X$ be a log curve over $S$.
    The \textit{vertex algebra bundle} $\cV$ \textit{associated to} $V$, is a sheaf of functors from $\LS_S$ to $\mathbf{Set}$ on $X_\et$,
    defined by
    $$\cV:=(V\otimes\pi_*\cO_{\Fra^\infty_X})^{\Autu^0\cO},$$
    where $\pi$ is defined in Remark  \ref{pi}. 
    Concretely, for $T\in\LS_S$, 
    locally on $U\in X_\et$,
    the space $\cV_T(U_T)$ is given by $$\cV_T(U_T):=((V\otimes\cO_T(T))
    \otimes\pi_*\cO_{\Fra^\infty_X(T)}(U_T))^{\Autu^0\cO(T)}.$$
    This suggests that $\cV_T$ can be viewed as a sheaf of vertex algebras on  $(X_T)_{\et}$.
    We provide an explicit description of the $\Autu^0\cO(T)$-action on $V\otimes\cO_T(T)$. Recall that an element of $\Autu^0\cO(T)$ can be written as a formal power series
    $z\mapsto f(z)=a_1z+a_2z^2+...$ where $a_1\in\cO_T(T)$ is a unit. 
    Given such an $f(z)$, 
    we can find $v_i\in \cO_T(T)$, $i\geq0$ such that
    $$f(z)=\exp\left(\sum_{i\geq0}v_iz^{i+1}\partial_z\right)v_0^{z\partial_z}\cdot z.$$ 
    We note that $z\partial_z\cdot z=z$, 
    and so $v_0^{z\partial_z}\cdot z=v_0z$.
    One has
    $$\exp\left(\sum_{i\geq0}v_iz^{i+1}\partial_z\right)v_0z=v_0z+\sum_{i>0}v_iv_0z^{i+1}+\frac{1}{2}\left(\sum_{i\geq0}v_iz^{i+1}\partial_z\right)\cdot\sum_{i>0}v_iv_0z^{i+1}+...$$
    By comparing the coefficients, 
    we obtain all $v_i$, $i\geq0$ from the system of equations
    \begin{equation*}
    	\begin{aligned}
    		v_0&=a_1\\
    		v_1v_0&=a_2\\
    		v_2v_0+v_1^2v_0&=a_3,\quad \text{etc.}
    	\end{aligned}
    \end{equation*}
    Then the action of $f(z)$ on $V\otimes\cO_T(T)$ is defined to be the linear operator
    \begin{equation}\label{action}
    	\exp\left(-\sum_{j>0}L_j\otimes v_j\right)(\id_V\otimes v_0^{-L_0}).
    \end{equation}
    We note that the action of $\id_V\otimes v_0^{-L_0}$ is given by
    $$(\id_V\otimes v_0^{-L_0})(v\otimes r)=v\otimes v_0^{-m}r$$
    for $v\in V_m$, $r\in\cO_T(T)$.
    The above operator is well-defined from the assumption that $V$ is bounded from below.

    Let $x:D_T\to X_T$ be a $D_T$-point over $T$.
    Recall that we have a trivial $\Autu^0\cO(T)$-torsor $\Aut_x$ on $D_x$. 
    Thus restricting it to $D_x^\times$, 
    we obtain a trivial $\Autu^0\cO(T)$-torsor on $D_x^\times$.
    We define the 
    \textit{fiber} $\cV_x$  
    \textit{of the vertex algebra bundle} 
    $\cV$  {over} $D_x$ associated to $V$,
    to be a sheaf over $D_T$,
    $$\cV_x=((V\otimes\cO_T(T))\otimes\pi_*\cO_{\Fra_{D_x}^\infty(T)})^{\Autu^0\cO(T)}.$$
    \begin{rk}
    	It follows that the space of global sections of $\cV_x$ is
    	$$((V\otimes\cO_T(T))\otimes\pi_*\cO_{\Fra^\infty_{D_x}(T)}(D_T))^{\Autu^0\cO(T)}
    	\simeq
    	((V\otimes\cO_T(T))\otimes\Aut_x)^{\Autu^0\cO(T)}.
    	$$
    	From Proposition \ref{thm:coor},
    	we see that if a coordinate is fixed,
    	then $\cV_x(D_T)$
    	can be identified with 
    	$V\otimes\cO_T(T)$.
    	This implies that the fibers of $\cV$ are all isomorphic to the vertex algebra $V$.
    \end{rk}
    
    \begin{definition}\label{connection}
    	A \textit{connection} on $\cV$ is a linear morphism
    	$$\nabla:\cV\to\cV\otimes\Omega_{X/S}^1$$
    	which, 
    	locally over $U\in X_\et$ and for $T\in\LS_S$
    	is given by a morphism
    	$$\nabla_T(U_T):\cV_T(U_T)\to\cV_T(U_T)\otimes\Omega_{X_T/T}^1(U_T)
    	$$
    	satisfies the following conditions:
    	\\ 
    	1)  $\nabla_T(U_T)$ is $\cO_T(T)$-linear;
    	\\
    	2) $\nabla_T(U_T)(f\cdot (v\otimes z))=v\otimes zdf+f\nabla_T(U_T)(v\otimes z)$ for $f\in\cO_{X_T}(U_T)$, 
    	where
    	$z$ is a local coordinate and
    	$d:\cO_{X_T}\to\Omega^1_{X_T/T}$ is the log derivative.
    \end{definition}
    
     For a coordinate $z\in\Aut_x(T)$, suppose that it factors through some $U_T\in(X_T)_\et$, i.e., there is a commutative diagram
     $$
     \begin{tikzcd}
     	D_T \arrow[rd, dashed] \arrow[r, "z"] \arrow[d] & X_T           \\
     	D^\circ_T \arrow[r]                             & U_T \arrow[u]
     \end{tikzcd}
     $$
     Then $L_{-1}\otimes\id+\id\otimes\partial_z\in\End \cV_T(U_T)$ locally describes a connection on $\cV_T$ near $x$, 
     which we denote by $\nabla_T$. Explicitly, it is given by 
     (\cite{damiolini}-2.7)
     $$\nabla_T(U_T)(v\otimes z^n)=L_{-1}v\otimes z^ndz+nv\otimes z^{n-1}dz.$$ 
     
     \begin{rk}\label{section}
     	As we noted previously, maps $D^\times\to X$, $D^\circ\to X$ and $D\to X$ cannot be \'{e}tale in general,
     	but differentials on these discs induced from that on $X$ still work for our construction.
     	For a sheaf $\cF$ on $(X_T)_{\et}$, we denote by $\Gamma(D_T^\times, \cF)$ 
     	the space of global sections of the pull-back of $\cF$ over $D_T^\times$.
     \end{rk}

     Recall in Theorem 
     \ref{logder} that
     the sheaf of log differentials 
     $\Omega^1_{D_T^\times/T}\simeq\cO_{D_T^\times}\otimes M_{D_T^\times}^{\operatorname{gp}}/(R_1+R_2)$ is generated as an $\cO_{D_T^\times}$-module by the image of $d:M_{D_T^\times}\to \Omega^1_{D_T^\times/T}$.
     So a section of $\cV_{D_x^\times}\otimes\Omega^1_{D_x^\times/T}$ can be written as $(v\otimes z)\otimes fdm$, where $v\in V$, $z\in\Aut_x$, $f\in\Gamma(D_T^\times,\cO_{X_T})$, $m\in \Gamma(D_T^\times,M_{X_T})$. 
     Indeed
     the fixed coordinate $z$, 
     locally on $\ul{T}=\Spec A$, 
     gives rise to an isomorphism $\ul{D^\times_T}\simeq\Spec A((z))$, 
     and then $f\in A((z))$. 
     Let $Q\to M_X(X)$ be a chart of $(X,M_X)$. 
     Since $(X,M_X)$ is fs, we suppose that $Q$ is generated by $e_1,..,e_n$.
     
     \begin{definition}
     	Let $T\in\LS_S$ and let $x$ be a $D_T$-point of $X_T$.
     	With respect to a fixed coordinate $z$,
     	the \textit{vertex operation} $(\cY^\vee_x)_T$ \textit{associated to} 
     	$Y: V\to\End V[[z^{\pm1}]]$ is defined to be the $\cO_T(T)$-linear map 
     	$$(\cY^\vee_x)_T:\Gamma(D_x^\times,\cV_{T}\otimes\Omega^1_{X_T/T})
     	\to\End\cV_x(D_T)$$
     	$$(v\otimes z)\otimes z^rd\bar{e}_i\mapsto \sum_{n\in\bZ}(e_i)_nv_{n+r}\otimes z$$
     	where $\bar{e}_i$ is the image of $e_i$ under the composition $Q\to M_X(X)\to M_{X_T}(X_T)\to \Gamma(D_T^\times,M_{X_T})$, and $d\bar{e}_i=\sum_n(e_i)_nz^ndz\in\Omega^1_{D^\times_T/T}(D_T^\times)$ with $(e_i)_n\in \cO_T(T)$. 
     \end{definition}
     
     After fixing a coordinate $z$, we observe that
     $$(\cY^\vee_x)_T
     (L_{-1}v\otimes z^rdz
     +v\otimes rz^{r-1}dz)
     =
     (L_{-1}v)_r+rv_{r-1}.$$
     It follows immediately from 
     the following Proposition \ref{vanish} that
     the above expression is $0$.
     Therefore $(\cY^\vee_x)_T$ vanishes on the \textit{total derivatives} $\nabla_T\Gamma(D_x^\times,\cV_{T})$ and hence factors through 
     $\Gamma(D_x^\times,\cV_{T}\otimes\Omega^1_{X_T/T})/\nabla_T\Gamma(D_x^\times,\cV_{T})$.
     
     \begin{prop}[\cite{frenkel-benzvi}-3.1.6]\label{vanish}
     	For all $a\in V$,
     	$Y(L_{-1}a,z)=\partial_z Y(a,z)$.
     \end{prop}
     
     The proof is not difficult but we will not use it in the sequel,
     and so is omitted.
    
    \subsection{Conformal blocks}
    
    To describe the space of conformal blocks, 
    we first construct the Lie algebra associated with the punctured space of $X$, which acts on the inserted modules at the framed points.
    
    With respect to a $D_T$-point $x:D_T\to X_T$, 
    the \textit{punctured space of} $X_T$ is defined to be the subscheme ${\mathop{X}\limits^{\circ}}_T=X_T-\im x$ of $X_T$,
    which is called the $x$-\textit{punctured subscheme of} $X_T$ in the present paper. 
    The \'{e}tale topology on ${\mathop{X}\limits^{\circ}}_T$ arises from the follows:
    \\ \hspace*{\fill} \\
    $\bullet$ Let $\{f_i:U_i\to X_T\}$ be an \'{e}tale covering of $X_T$, 
    and then $\left\{f_i\big|_{f_i^{-1}({\mathop{X}\limits^{\circ}}_T)}:f_i^{-1}({\mathop{X}\limits^{\circ}}_T)\to {\mathop{X}\limits^{\circ}}_T\right\}$ is an \'{e}tale covering of ${\mathop{X}\limits^{\circ}}_T$.
    \\ \hspace*{\fill} \\
    In the case of puncturing multiple points $x_1,...,x_n:D_T\to X_T$, we set $\circX_j$ as $x_j$-punctured subscheme of $X_T$, 
    and then the \'{e}tale covering of ${\mathop{X}\limits^{\circ}}_T$ associated to $\{f_i:U_i\to X_T\}$ is defined as $$\left\{f_i\big|_{f_i^{-1}(\circX_1\cap...\cap\circX_n)}:f_i^{-1}(\circX_T)\to\circX_T\right\}.$$

    As we stated at the beginning, 
    our aim now is to define a Lie algebra
    (proved later in Proposition \ref{lieD}) 
    associated with the vertex algebra $V$, 
    which acts on the sheaf of $V$-modules. 
    Prior to this, 
    we introduce a space attached to the punctured curve $\circX$, 
    through which the Lie algebra acts.
    
    \begin{definition}
    	In the case that $\circX$ is obtained by removing exactly one framed point,
    	i.e.,
    	${\mathop{X}\limits^{\circ}}_T=X_T-\im x$
    	for a $D_T$-point $x$,
    	define 
    	$$\Lie_\circX(V)_T:=\frac{\cV_T(\circX_T)\otimes\Omega_{X_T/T}^1(\circX_T)}{\nabla_T\cV_T(\circX_T)}.$$
    \end{definition}

    \begin{rk}
    	The action of $\nabla_T$ on $\cV_x$ is induced by its action on $\cV_T(X_T)$,
    	which is of the same form
    	$L_{-1}\otimes\id+\id\otimes \partial_z$
    	for a fixed coordinate $z$ at $x$.
    	Indeed the connection $\nabla_T$ gives rise to a morphism
    	$$\cV_x\to\cV_x\otimes\Omega_{D_T /T}^1$$
    	which satisfies the conditions in Definition \ref{connection}.
    	If we denote by $i:D^\times_T\rightarrow D_T$ the natural morphism
    	whose associated homomorphism of complete rings is the inclusion
    	$\cO_T[[t]]\hookrightarrow \cO_T((t))$
    	and the morphism of underlying topological spaces is the inclusion map,
    	then for a $D_T$-point $x$ on $X_T$ (often denoted by $D_x$),
    	we have a morphism $\hat{x}:D_T^\times \to X_T$
    	by composing $x:D_T\to X_T$ with $i$.
    	Let $D_x^\times$ denote the above morphism $\hat{x}$.
    	Since $\cV_x$ is a vertex algebra bundle over $D_T$ associated to $D_x$,
    	its pullback
    	$\cV_{\hat{x}}:=i^*\cV_x$
    	is a vertex algebra bundle over $D^\times_T$
    	associated to $V$.
    	Therefore we yield a morphism 
    	$$\cV_{\hat{x}}\to \cV_{\hat{x}}\otimes\Omega_{D_T^\times /T}^1$$
    	induced by $\nabla_T$.
    	We are now in a position to define
    	$$\Lie_{D^\times}(V)_T=
    	\frac{\cV_{\hat{x}}(D_T^\times)\otimes\Omega_{D^\times_T /T}^1(D_T^\times)}{\nabla_T\cV_{\hat{x}}(D_T^\times)}.$$
    \end{rk}

    The morphism $x^\circ:D^\circ_T\to X_T$ associated to $x$ admits a lifting $\hat{x}^\circ:D^\times_T\to \circX_T$ such that the diagram
    $$
    \begin{tikzcd}
    	D^\times_T \arrow[d, hook] \arrow[r, "\hat{x}^\circ"] & {\mathop{X}\limits^{\circ}}_T \arrow[d] \\
    	D^\circ_T \arrow[r, "x^\circ"]                        & X_T                                    
    \end{tikzcd}
    $$
    commutes. 
    We then pull $\Lie_\circX(V)_T$ back through $\hat{x}^\circ$ to $D_T^\times$,
    identified with the image of the restriction map
    \begin{equation}\label{liemap}
    	\Lie_\circX(V)_T\to\Lie_{D^\times}(V)_T
    \end{equation}
    which is denoted by $\hat{x}^*\Lie_\circX(V)_T$.
    Concretely, the above restriction map can be described by the following two maps:
    \\
    (1) $\cV_T(\circX_T)\to\cV_{\hat{x}}(D^\times_T)$ given by taking the stalk of a section near the punctured point $x$.
    Specifically 
    it is induced by the map
    $$\pi_*\cO_{\Fra^\infty_{X}(T)}(\circX_T)\to
    i^*\pi_*\cO_{\Fra^\infty_{D_x}(T)}(D^\times_T),\quad
    \xi\mapsto (\hat{x}^\circ)^\sharp\xi;$$
    (2) $\Omega_{X_T/T}^1(\circX_T)\to
    \Omega_{D^\times_T /T}^1(D_T^\times)$ induced by
    $(\hat{x}^\circ)^*\Omega_{X_T/T}^1\to
    \Omega_{D^\times_T /T}^1$.
    
    \begin{prop}\label{lieD}
    	There is a natural Lie algebra structure on 
    	$\Lie_{D^\times}(V)_T$.
    \end{prop}

    \begin{proof}
    	The proof is based on \cite{frenkel-benzvi}-4.1.2.
    	By definition, 
    	we have
    	\begin{equation*}
    		\begin{aligned}
    			\cV_{{\hat{x}}}(D^\times_T)
    			=i^*\cV_x(D^\times_T)
    			&=\cV_x(D_T)\otimes_{\cO_T(D_T)}\cO_{D^\times_T}(D^\times_T)
    			\\
    			&
    			\simeq((V\otimes\cO_T(T))\otimes\Aut_x(T))^{\Aut^0\cO(T)}\otimes
    			_{\cO_T(D_T)}\cO_{D^\times_T}(D^\times_T).
    		\end{aligned}
    	\end{equation*}
        If we fix a coordinate $z$,
        then it follows that there are $\cO_T(T)$-module isomorphisms
        $$\cV_{{\hat{x}}}(D^\times_T)\otimes\Omega_{D^\times_T /T}^1(D^\times_T)
        \simeq
        V\otimes\cO_T(T)((z))dz
        \simeq V\otimes\cO_T(T)((z))
        $$
        where the latter isomorphism is induced by the valuation of a differential at $\partial_z$.
        Then the connection $\nabla_T$ near $D_x$
        induces an action of the form 
        $L_{-1}\otimes\id +\id\otimes\partial_z$
        on $V\otimes\cO_T(T)((z))$.
        Thus we have
        $$\Lie_{D^\times}(V)_T
        \simeq
        \frac{V\otimes\cO_T(T)((z))}{\im\nabla_T}.$$
        Observe that the commutative $\cO_T(T)$-algebra $\cO_T(T)[z^{\pm1}]$
        carries a vertex algebra structure
        by Example \ref{comm}.
        Therefore we obtain that the tensor product of $V$ and $\cO_T(T)[z^{\pm1}]$ carries a vertex algebra
        structure
        by Proposition \ref{tensor}.
        By the following Lemma
        \ref{fb413},
        it follows that
        $$
        U(V\otimes\cO_T[z^{\pm1}])_0=
        \frac{V\otimes\cO_T(T)[z^{\pm1}]}{\im(\nabla_T|_{V\otimes\cO_T(T)[z^{\pm1}]})}
        $$
        admits a Lie algebra structure.
        Moreover, 
        the Lie algebra $U(V\otimes\cO_T[z^{\pm1}])_0$
        is a topological Lie algebra
        where the basis of open neighborhoods of $0$ consists of the subspaces indexed by $M>0$ spanned by
        $\sum_{n\geq M}f_nA_{[n]}$, for $f_n\in\cO_T(T)$, $v\in V$.
        The Lie bracket (Remark \ref{jacobi}) on $U(V\otimes\cO_T[z^{\pm1}])_0$
        induced by
        \begin{equation}\label{lieDeq}
        	[A\otimes z^m,B\otimes z^k]=\sum_{n\geq0}
        	\left(
        	\begin{matrix}
        		m\\
        		n
        	\end{matrix}\right)
        	(A_{n}B)\otimes z^{m+k-n}
        \end{equation}
        is continuous,
        and 
        ${V\otimes\cO_T(T)((z))}/{\im\nabla_T}$
        is the completion of
        $U(V\otimes\cO_T[z^{\pm1}])_0$
        with respect to this topology.
        This allows us to extend the Lie bracket to $\Lie_{D^\times}(V)_T$
        by continuity.
        Thus we obtain a Lie algebra structure on $\Lie_{D^\times}(V)_T$.
    \end{proof}

    \begin{lem}[\cite{frenkel-benzvi}-4.1.3]\label{fb413}
        Let $(V,\0,T,Y)$ be a vertex algebra,
        $U(V)_0$ be the quotient $V/\im T$
        and denote the image of $A\in V$ in $U(V)_0$ by $A_{[0]}$. 
        Then the linear map $U(V)_0^{\otimes 2}\to U(V)_0$
        defined by 
        $$A_{[0]}\otimes B_{[0]}\mapsto
        (A_{(0)}B)_{[0]}$$
        defines a Lie algebra structure on
        $U(V)_0$.
    \end{lem}

    \begin{proof}
    	From the following skew-symmetry formula,
    	it follows that
    	$$A_{0}B=-B_{0}A 
    	\mod T.$$
    	This implies the anti-commutativity
    	and
    	alternating property.
    	To prove the Jacobi identity,
    	let $A,B,C\in V$,
    	then one has that
    	\begin{equation*}
    		\begin{aligned}
    			&[[A_{[0]},B_{[0]}],C_{[0]}]+[[C_{[0]},A_{[0]}],B_{[0]}]+[[B_{[0]},C_{[0]}],A_{[0]}]
    			\\
    			&=
    			[(A_0B)_{[0]},C_{[0]}]+[(C_0A)_{[0]},B_{[0]}]+[(B_0C)_{[0]},A_{[0]}]
    			\\
    			&=((A_0B)_0C)_{[0]}+((C_0A)_0B)_{[0]}+((B_0C)_0A)_{[0]}
    			\\
    			&=([A_0,B_0]C)_{[0]}+([C_0,A_0]B)_{[0]}+([B_0,C_0]A)_{[0]}
    			\\
    			&=(A_0B_0C)_{[0]}-(B_0A_0C)_{[0]}+(C_0A_0B)_{[0]}-(A_0C_0B)_{[0]}+(B_0C_0A)_{[0]}-(C_0A_0B)_{[0]}
    			\\
    			&=2((A_0B_0C)_{[0]}+(C_0A_0B)_{[0]}+(B_0C_0A)_{[0]}).
    		\end{aligned}
    	\end{equation*}
        where the third equality follows from
        $[A_0,B_0]=(A_0B)_0$ 
        (by Remark \ref{jacobi}) and the last equality follows from
        $A_0B=-B_0A\mod T$.
        Moreover, it follows that
        \begin{equation*}
        	\begin{aligned}
        		&(A_0B_0C)_{[0]}+(C_0A_0B)_{[0]}+(B_0C_0A)_{[0]}
        		\\
        		&=[A_{[0]},(B_0C)_{[0]}]+[B_{[0]},(C_0A)_{[0]}]+[C_{[0]},(A_0B)_{[0]}]
        		\\
        		&=[A_{[0]},[B_{[0]},C_{[0]}]]+[B_{[0]},[C_{[0]},A_{[0]}]]+[C_{[0]},[A_{[0]},B_{[0]}]]
        		\\
        		&=-([[A_{[0]},B_{[0]}],C_{[0]}]+[[C_{[0]},A_{[0]}],B_{[0]}]+[[B_{[0]},C_{[0]}],A_{[0]}])
        	\end{aligned}
        \end{equation*}
        Combining the above two equations,
        we obtain 
        $$[[A_{[0]},B_{[0]}],C_{[0]}]+[[C_{[0]},A_{[0]}],B_{[0]}]+[[B_{[0]},C_{[0]}],A_{[0]}]=0$$
        which is the Jacobi identity.
        This completes the proof.
    \end{proof}
    
    \begin{prop}[\textbf{Skew-symmetry} \cite{frenkel-benzvi}-3.2.5]
    	Let $V$ be a vertex algebra with translation operator $T$.
    	Then
    	$$Y(A,z)B=e^{zT}Y(B,-z)A$$
    	for $A,B\in V$.
    \end{prop}
    
    The proof is not particularly difficult but will not be reproduced here,
    as we do not need the details.
    
    \begin{cor}
    	$\hat{x}^*\Lie_{\circX}(V)_T$
    	is a Lie subalgebra of $\Lie_{D^\times}(V)_T$.
    \end{cor}

    \begin{proof}
    	This immediately follows from that the image of $(\hat{x}^\circ)^*\Omega_{X_T/T}^1\to
    	\Omega_{D^\times_T /T}^1$ 
    	consists of finite series
        and 
        the space of such series is closed under the Lie bracket (\ref{lieDeq}).
    \end{proof}
    
    \begin{definition}
    	Let $M$ be a conformal $V$-module
    	on which the eigenvalues of $L_0^M$ are integers.
    	Define the sheaf $\cM$ of functors from $\LS_S$ to $\mathbf{Set}$,
    	on $X_\et$ 
    	{associated to} $M$
    	by
    	$$\cM(U):=(M\otimes\pi_*\cO_{\Fra^\infty_X}(U))^{\Autu^0\cO}$$
    	whose value at $T\in\LS_S$ is
    	$$\cM_T(U_T)=((M\otimes\cO_T(T))\otimes\pi_*\cO_{\Fra^\infty_X(T)}(U_T))^{\Autu^0\cO(T)}.$$
    	Similar to what we have done for conformal vertex algebras, 
    	the fiber $(\cM_T)_x$ at $x$ is defined to be
    	$$\cM_x:=((M\otimes\cO_T(T))\otimes\Aut_x(T))^{\Autu^0\cO(T)}.$$
    \end{definition}
    
    \begin{rk}
    	We explain why we assume that the eigenvalues of $L_0$
    	on $M$ are integers.
    	Recall in formula (\ref{action})
    	that there appears a term 
    	$v_0^{-L_0}$
    	acting on $M$
    	by
    	$v_0^{-L_0}m=v_0^{-l}m$
    	where $L_0\cdot m=lm$.
    	If $l$ is not an integer,
    	then the scalar $v_0^{-l}$ is not necessarily well defined.
    	For instance,
    	if the base scheme has underlying scheme
    	$\Spec\bC$,
    	then a complex number may have multiple roots,
    	and thus
    	$v_0^{-l}m$ is no longer well-defined.
    \end{rk}
    
    \begin{rk}
    	Recall that
    	$\cV_x=((V\otimes\cO_T(T))\otimes\pi_*\cO_{\Fra^\infty_{D_x}(T)})^{\Autu^0\cO(T)}$
    	is a sheaf on $D_T$
    	centered at $x$,
    	and its space of global sections is
    	\begin{equation*}
    		\begin{aligned}
    			\cV_x(D_T)&=
    			((V\otimes\cO_T(T))\otimes\pi_*\cO_{\Fra_{D_x}^\infty(T)}(D_T))^{\Autu^0\cO(T)}
    			\\
    			&=((V\otimes\cO_T(T))\otimes\Aut_x(T))^{\Autu^0\cO(T)}
    		\end{aligned}
    	\end{equation*}
        Since in the sequel any sheaf associated to a $V$-module $M$
        only plays its role locally,
        i.e., as a fiber,
        so for ease of notation we define $\cM_x$
        to be the space of sections rather than a sheaf.
    \end{rk}
    
    In the 1-point case, 
    the fiber $\cM_x$ admits an action of $\Lie_\circX(V)_T$ through $(\cY_{M,x}^\vee)_T$ associated to the vertex operation $Y_M: V\to\End M[[z^{\pm1}]]$. 
    Namely, the $\Lie_{\circX}(V)_T$-action is given by
    $$\hat{x}^*\Lie_\circX(V)_T=
    \im\left(
    \Lie_\circX(V)_T\to\Lie_{D^\times}(V)_T\right)
    \stackrel{(\cY_{M,x}^\vee)_T}{\longrightarrow}\End\cM_x.$$
    In the rest of the present paper,
    for the base change scheme $T\in\LS_S$,
    by an $\cO_T$-module,
    we mean an $\cO_T(T)$-module when the target algebraic structure does not depend on the topology.
    
    \begin{definition}
    	The \textit{space of coinvariants} (or \textit{space of covacua}) $H_V(X,x,M)_T$ is defined to be the $\cO_T$-module  $$\cM_x/\Lie_\circX(V)_T\cdot\cM_x.$$
    	Its dual, the \textit{space of conformal blocks} $C_V(X,x,M)_T$, is defined to be  $$\Hom_{\Lie_\circX(V)_T}(\cM_x,\cO_T(T)).
    	$$
    \end{definition}
    
    In the case of multiple $D_T$-points $x_1,...,x_n$, with an $n$-tuple $M_1,...,M_n$ of conformal $V$-modules attached to these points respectively, 
    let $\circX=X-\{\im x_1,...,\im x_n\}$.
    The $\Lie_\circX(V)_T$-action on $\bigotimes_i\cM_{x_i}$ is induced by 
    \begin{equation*}
    	\begin{aligned}
    		\Lie_\circX(V)_T=
    		\frac{\cV_T(\circX_T)\otimes\Omega_{X_T/T}^1(\circX_T)}{\nabla_T\cV_T(\circX_T)}&
    		\stackrel{}{\to}
    		\bigoplus_{i=1}^n\Lie_{D^\times}(V)_T=
    		\bigoplus_{i=1}^n
    		\frac{\cV_{{\hat{x}_i}}(D_T^\times)\otimes\Omega_{D^\times_T /T}^1(D_T^\times)}{\nabla_T\cV_{{\hat{x}_i}}(D_T^\times)}
    		\\
    		v\otimes\mu&\mapsto \left(v\otimes\mu|_{D_{x_1}^\times}
    		,...,
    		v\otimes\mu|_{D_{x_n}^\times}\right)
    	\end{aligned}
    \end{equation*}
    from the map (\ref{liemap}),
    and 
    $$\bigotimes_i(\cY^\vee_{M_i,x_i})_T:
    \bigoplus_i\Lie_{D^\times}(V)_T
    \to\bigotimes_i\End(\cM_i)_{x_i}$$ 
    where
    $\mu|_{D_{x_i}^\times}$ ($i=1,...,n$)
    denotes the pull-back of $\mu$ via
    $D_{x_i}^\times\to \circX_T$.
    If we denote
    the image of the above map
    $\Lie_\circX(V)_T
    \stackrel{}{\to}
    \bigoplus_i\Lie_{D^\times}(V)_T$
    by
    $\hat{x}_\circ^*\Lie_\circX(V)_T$,
    then as a consequence of the one-point case,
    the vertex operation
    $\bigotimes_i(\cY^\vee_{M_i,x_i})_T$
    vanishes on the total differentials and thus induces an action of $\hat{x}_\circ^*\Lie_\circX(V)_T$ on
    $\otimes_i\cM_{x_i}$.
    Likewise, the \textit{space of coinvariants} is defined to be the $\cO_T$-module  $$H_V(X,(x_i),(M_i))_T:=\bigotimes_i(\cM_i)_{x_i}/\Lie_\circX(V)_T\cdot\bigotimes_i(\cM_i)_{x_i}.$$
    Explicitly, for $\sigma\in\Lie_{\circX}(V)_T$ and $\bigotimes_i m_i\in\bigotimes_i(\cM_i)_{x_i}$, 
    the associated action is given as
    $$\sigma\cdot \bigotimes_i m_i=\sum_im_1\otimes\cdots\otimes (\cY^\vee_{x_i})_T(\sigma|_{D_{x_i}^\times})m_i\otimes\cdots\otimes m_n.$$
    The associated \textit{space of conformal blocks} is defined to be the dual of the space of coinvariants $$C_V(X,(x_i),(M_i))_T:=\Hom_{\Lie_\circX(V)_T}\left(\bigotimes_i(\cM_i)_{x_i},\cO_T(T)\right).$$
    
    We now turn to the functoriality of conformal blocks.
    For morphisms $M_i'\to M_i$, $i=1,...,n$ of conformal $V$-modules,
    it is straightforward to see that there is a corresponding map of spaces of conformal blocks
    $$C_V(X,(x_i),(M_i))_T\to C_V(X,(x_i),(M_i'))_T.$$
    In addition, 
    the space of conformal blocks is also functorial in the underlying vertex algebra.
    
    \begin{prop}
    	Let $V\to W$ be a homomorphism of conformal vertex algebras. Then for any $n$-tuple of conformal $W$-modules $(M_i)_{i=1,...,n}$, there is a natural morphism
    	$$C_W(X,(x_i),(M_i))_T\to C_V(X,(x_i),(M_i))_T.$$
    \end{prop}

    \begin{proof}
    	It suffices to prove the 1-point case. Let $\phi:V\to W$ denote the homomorphism of conformal vertex algebras. Then a conformal $W$-module $M$ carries a conformal $V$-module structure
    	via composing
    	$\phi$ and $Y_M:W\to\End M[[z^{\pm1}]]$. 
    	There is a morphism of vertex algebra bundles $\cV\to\cW$ induced by $\phi\otimes \id$ in the sense of
    	$\cV(U)=(V\otimes\pi_*\cO_{\Fra^\infty_X}(U))^{\Autu^0\cO}$.
    	Thus we obtain a morphism $\cV_{T}\otimes\Omega^1_{X_T/T}
    	\to
    	\cW_{T}\otimes\Omega^1_{X_T/T}$. If we denote by $\bar{v}$ the image of $v\in V$ in $W$, then since $\overline{L(z)v}=\bar{L}(z)\bar{v}$, 
    	we obtain a well-defined morphism $\Lie_\circX(V)_T\to\Lie_\circX(W)_T$. Therefore there is a morphism of spaces of coinvariants
    	$$H_V(X,x,M)_T\to H_W(X,x,M)_T$$
    	induced by  $\Lie_\circX(V)_T\to\Lie_\circX(W)_T$. 
    	This naturally induces a morphism of spaces of conformal blocks as their duals.
    \end{proof}

    The following theorem is a weak version of the propagation of vacua.

    \begin{thm}
    	Let $x,y$ be two distinct $D_T$-points of $X_T$, and 
    	let $M$ be a conformal $V$-module with integral eigenvalues with respect to the action of $L_0$. Let $\circX_y$ denote the $y$-punctured space of $X$.
    	If $\cV_y(D_T)$ is an irreducible representation of $\hat{y}^*\Lie_{\circX_y}(V)_T$ through $(\cY_y^\vee)_T$, then we have a monomorphism of spaces of conformal blocks
    	$$C_V(X,(x,y),(M,V))_T
    	\hookrightarrow
    	C_V(X,x,M)_T.$$
    	Consequently, if $x_1,...,x_n$ are distinct $D_T$-points, with attached conformal $V$-modules $M_1,...,M_n$ respectively, on which the eigenvalues of $L_0$ are integral, then there is a monomorphism
    	$$C_V(X,(x_1,...,x_n,y),(M_1,...,M_n,V))_T
    	\hookrightarrow C_V(X,(x_1,...,x_n),(M_1,...,M_n))_T.$$
    \end{thm}
    
    \begin{proof}
    	Let $\circX_x$ denote the $x$-punctured space of $X$, 
    	and $\circX$ the punctured space of $X$ by both $x$ and $y$. 
    	From the inclusions
    	$\circX\hookrightarrow\circX_y$ and 
    	$\circX\hookrightarrow\circX_x$,
    	we have the corresponding restriction maps
    	$\cV_T((\circX_x)_T)\to\cV_T(\circX_T)$
    	and
    	$\Omega_{X_T/T}^1((\circX_y)_T)\to
    	\Omega_{X_T/T}^1(\circX_T)$.
    	For $\eta\in\cV_T((\circX_x)_T)$
    	such that $\nabla_T\xi=\eta$ for some
    	$\xi$,
        let $\bar{\xi},\bar{\eta}\in\cV_T(\circX_T)$
        denote the image of $\xi$ and
        $\eta$ under the restriction map respectively.
        Then clearly we have
        $\bar{\eta}=\nabla_T\bar{\xi}=0$ in $\cV_T(\circX_T)$.
        Thus we obtain a well-defined map
        $\Lie_{\circX_y}(V)_T\to
        \Lie_{\circX}(V)_T$
        induced from the above two restriction maps.
        
        If we fix coordinates at $x$ and $y$ to be $z_x$ and $z_y$ respectively,
        then 
        $$\cM_x=
        ((M\otimes\cO_T(T))\otimes\Aut_x(T))^{\Autu^0\cO(T)}$$
        is trivialized to $M\otimes\cO_T(T)$.
        Likewise, $(\cV_T)_y$ is trivialized to $V\otimes\cO_T(T)$.
        Define a map
        \begin{equation*}
        	\begin{aligned}
        		\Hom_{\Lie_\circX(V)_T}(\cM_x\otimes(\cV_T)_y,\cO_T(T))
        		&\to
        		\Hom_{\Lie_{\circX_x}(V)_T}(\cM_x,\cO_T(T))\\
        		\vp&\mapsto\bar{\vp}:m\mapsto\vp(m\otimes\0)
        	\end{aligned}
        \end{equation*}
        Note that we often omit the coefficient ring $\cO_T(T)$ in the rest of the proof,
        as well as the definition of $\bar{\vp}$ in the above map.
        Since any $\nu\in\Lie_{\circX_x}(V)_T$
        is regular outside $x$,
        it is also regular on $\circX$,
        and thus
        $\nu\in\Lie_{\circX}(V)_T$.
        This implies that the above map $\vp\mapsto\bar{\vp}$ is well-defined.
        Suppose that there are two conformal blocks $\vp,\psi\in C_V(X,(x,y),(M,V))_T$
        satisfying  $\bar{\vp}=\bar{\psi}$.
        For $v\in \cV_y(D_T)$,
        since $\cV_y(D_T)$ is 
        an irreducible representation of $\hat{y}^*\Lie_{\circX_y}(V)_T$,
        there exists some
        $\mu\in\Lie_{\circX_y}(V)_T$
        such that $v=\mu_y\cdot\0$
        where $\mu_y\in\hat{y}^*\Lie_{\circX_y}(V)_T$ is the image of $\mu$ in
        $\Lie_{D^\times}(V)_T$.
        As $\mu$ is regular at $x$ and the only possible pole on $X_T$ is $y$,
        we obtain that $\mu\in  \Lie_{\circX}(V)_T$.
        And this gives rise to an action of $\mu$ on 
        $(\cM_T)_x\otimes(\cV_T)_y$.
        Then for $m\in M\otimes\cO_T(T)$, 
        it follows that
        \begin{equation*}
        	\begin{aligned}
        		\vp(m\otimes v)
        		=
        		\vp(m\otimes\mu_y\cdot\0)
        		&=
        		-\vp(\mu_x\cdot m\otimes\0)
        		\\
        		&=-\psi(\mu_x\cdot m\otimes\0)
        		=\psi(m\otimes\mu_y\cdot\0)
        		=\psi(m\otimes v)
        	\end{aligned}
        \end{equation*}
        where the second equality follows from the definition of conformal blocks.
        This proves that $\vp=\psi$ and hence the map $\vp\mapsto\bar{\vp}$ is injective.
    	The $n$-point case statement follows from induction on $n$.
    \end{proof}

    \section{Normal crossing}
    
    We provide an example that illustrates our theory, which is also a typical construction in log geometry.
    This example will be discussed in two cases,
    depending on the log structure on the base scheme.
    Everything done in this section is over $\bC$. 
    For ease of notation, 
    we omit the subscript of the base change, 
    i.e., $X_\bC$ is denoted by $X$ for any (formal) log scheme $X$. 
    
    The first case is to take the base scheme to be 
    $S=\Spec(\bN\to\bC)$
    induced by the monoid morphism
    $\bN\to\bC$,
    $n\mapsto 0^n$.
    The log scheme $S$ is said to be a
    \textit{standard log point over $\bC$}
    in \cite{ogus}-III-1.5.2.
    The reason for choosing this log structure is that
    every fine log point over $\bC$ is dominated by
    $S$,
    in the sense of the following proposition.
    \begin{prop}[\cite{ogus}-III-1.5.5]
    	If $S'$ is any fine log point,
    	there exists an algebraically closed field $k$ and a morphism
    	$\Spec(\bN\to k)\to S'$.
    \end{prop}

    The other case is to endow the base scheme 
    $\underline{S}=\Spec\bC$
    with the trivial log structure
    $0\to \bC$.
    Unfortunately, 
    the trivial log scheme $\Spec(0\to \bC)$ behaves badly,
    and we will explain why.
    We start with introducing a useful criterion for 
    smoothness and étaleness.
    
    \begin{thm}[\cite{ogus}-IV-3.1.13]\label{etalesmooth}
    	Let 
    	$f: X \to Y$ 
    	be a morphism of log schemes, 
    	over a characteristic $0$ field $k$, 
    	admitting
    	a coherent chart 
    	$\theta: P\to Q$,
    	and let $x$ be a point of $X$,
    	satisfying the following
    	conditions:
    	\\
    	1) The kernel and the
    	cokernel 
    	(resp. the torsion part of the cokernel) 
    	of 
    	$\theta$ 
    	are finite groups
    	whose orders are both invertible in $k(x)$;
    	\\
    	2) The morphism of schemes
    	$\ul{X}\to
    	\ul{Y}\times_{\Spec k[P]}\Spec k[Q]$
    	is étale 
    	(resp. smooth) in some neighborhood of $x$.
    	\\
    	Then $f: X \to Y$ 
    	is étale (resp. smooth) in some neighborhood of $x$.
    \end{thm}

    \subsection{Case I:
    	base field $\bN\to\bC$}\label{casei}
    
    Let $R=\bC[x,y]/(xy)$, 
    and let $\circX=\Spec(\bN^2\to R)$ given by the chart $\alpha_{\circX}:\bN^2\to R$, $(m,n)\mapsto x^my^n$. 
    Let $S=\Spec(\bN\to \bC)$, with log structure induced by 
    $\bN\to\bC$, $n\mapsto 0^n$.
    The structure morphism $\circX\to S$ is given by the following log ring map
    $$
    \begin{tikzcd}
    	\bC \arrow[r, hook]     & {\bC[x,y]/(xy)} &  & 0^n \arrow[r, maps to]                  & x^ny^n                     \\
    	\bN \arrow[r] \arrow[u] & \bN^2 \arrow[u] &  & n \arrow[u, maps to] \arrow[r, maps to] & {(n,n)} \arrow[u, maps to]
    \end{tikzcd}
    $$
    From the construction of the module of log differentials in Section \ref{logsch} after Theorem \ref{logder},
    the module of log differentials $\Omega_{\circX/S}^1(\circX)$
    is isomorphic to
    $$
    \frac{\Omega_{R/\bC}^1\oplus(R\otimes\bZ^2)}{\langle
    (d\alpha_{\circX}(m),-\alpha_{\circX}(m)\otimes m),
    (0,1\otimes (n,n)),\ m\in\bN^2,n\in\bZ\rangle}.
    $$
    Observe that $\Omega_{R/\bC}^1$ is generated by $dx$ and $dy$,
    with $dx=d\alpha_{\circX}(1,0)$
    and
    $dy=d\alpha_{\circX}(0,1)$.
    Thus under the equivalence relation,
    the usual differential space
    $\Omega_{R/\bC}^1$ is
    ``included'' in 
    $R\otimes\bZ^2$,
    with relations (in the form (\ref{logdiff}))
    $$1\otimes (1,0)=d(1,0)=\frac{dx}{x},\quad 
    1\otimes (0,1)=d(0,1)=\frac{dy}{y}.$$
    In addition, 
    we have the relation
    $$0=1\otimes(1,1)=d(1,0)+d(0,1)=\frac{dx}{x}+\frac{dy}{y}.$$
    In summary,
    the space of global log differentials
    $\Omega_{\circX/S}^1(\circX)$
    is an $R$-module generated by
    $dx/x$ and $dy/y$
    with the relation
    $dx/x+dy/y=0$. 
    With respect to the base change
    $S\to\Spec(0\to\bZ)$,
    the discs $D^\times$ and $D^\circ$ (resp. $D$) are endowed with log structures induced by the charts $\bN\to\bC((t))$ and $\bN^2\to\bC[[t]]$ with $m\mapsto 0^m$ and $(m,n)\mapsto t^m0^n$ respectively.
    
    Let $X\subset\bP^2_\bC$ be the log scheme whose underlying scheme is determined by the homogeneous equation $xy=0$ with coordinates $(x,y,z)$ of $\bP^2_\bC$,
    endowed with the log structure induced by the embedding $\ul\circX\hookrightarrow \ul X$.
    Denote by $\infty_1=(1,0,0)$ and $\infty_2=(0,1,0)$. 
    Then it is easy to see that if we  remove these two points $\infty_{1,2}$ on $X$, we obtain the log scheme $\circX$, which is exactly the $z\not=0$ patch of $X$. 
    We denote by $D_1^\times$ and $D_2^\times$ the punctured discs centered at $\infty_1$ and $\infty_2$ respectively. 
    \begin{figure}[htbp]
    	\begin{center}
    		\begin{tikzpicture}
    			\draw [black](0,0) circle[radius=1];
    			\fill   (-1,0) circle[radius=1pt]  node [left,font=\small] {$\infty_1$};
    			\draw [black](2,0)circle[radius=1];
    			\fill   (3,0) circle[radius=1pt]  node [right,font=\small] {$\infty_2$};
    			\fill   (1,0) circle[radius=1pt]  node [right,font=\small] {$o$};
    		\end{tikzpicture}
    	\end{center}\caption{$xy=0$ in $\bP^2$}
    \end{figure}
    
    Let us determine the log structure on $X$.
    For an arbitrary monoid $P$,
    we write $\overline{P}=P/P^*$.
    For the node $o=(0,0,1)$, it is easy to see that $\overline{M}_{X,o}\simeq
    \overline{\bN^2\oplus R^*_o}\simeq\bN^2$. 
    If we have a $D$-point on $\circX$ such that the base point of $D$ is mapped to $o$,
    then we assume this map is given by $R\to\bC[[t]]$, $x\mapsto t$, $y\mapsto0$.
    In fact, if $x\mapsto t$, $y\mapsto g(t)$, then $0=xy\mapsto tg(t)$.
    Since $\bC[[t]]$ is an integral domain,
    we may take $g(t)=0$.
    Moreover, 
    the log structure is fitted  into the map through
    $\id_{\bN^2}$.
    Therefore such a point also gives rise to a coordinate at $o$.
    \begin{rk}
    	We now explain that the above log ring map corresponding to $D\to\circX$, 
    	which is denoted by $\theta_o$ in the sequel, given by
    	$$
    	\begin{tikzcd}
    		R \arrow[r, "\theta_o^\sharp"]                                             & {\bC[[t]]}                   &  & x^my^n \arrow[r, maps to]                     & t^m0^n                     \\
    		\bN^2 \arrow[r, "\theta_o^\flat=\id_{\bN^2}"] \arrow[u, "\alpha_{\circX}"] & \bN^2 \arrow[u, "\alpha_D"'] &  & {(m,n)} \arrow[u, maps to] \arrow[r, maps to] & {(m,n)} \arrow[u, maps to]
    	\end{tikzcd}
        $$
        is formally étale, 
        and hence gives rise to a coordinate at $o$.
        Let us decompose 
        $\theta_o:D\to \circX$ into
        $D\stackrel{j}{\longrightarrow}
        \Spec(\bN^2\to\bC[t])
        \stackrel{\theta_o^0}{\longrightarrow}
        \circX$,
        where 
        $\Spec(\bN^2\to\bC[t])$
        is induced by the monoid morphism
        $(m,n)\mapsto t^m0^n$,
        and $j$ is the natural map.
        It is clear that condition (1) in Theorem \ref{etalesmooth} is
        satisfied for the identity map.
        For condition (2),
        let $A$ be a commutative $\bC$-algebra with a square zero ideal $J$,
        and suppose there is a commutative diagram
        $$
        \begin{tikzcd}
        	{R\otimes_{\bC[\bN^2]}\bC[\bN^2]} \arrow[d, "s"'] \arrow[r] & {\bC[t]} \arrow[d] \\
        	A \arrow[r, "p"]                                            & A/J                 
        \end{tikzcd}
        $$
        where $p$ is the quotient map.
        Define $h:\bC[t]\to A$ by
        $h(a(t))=s(a(x)\otimes1)$.
        It is straightforward to check that $h$
        makes the two triangles in the diagram commutative.
        Moreover 
        $\Omega^1_{\bC[t]/R\otimes_{\bC[\bN^2]}\bC[\bN^2]}=0$ 
        implies the unramification.
        Thus the associated morphism of schemes
        is étale,
        and then the log scheme morphism 
        $\theta^0_o$ is étale.
        On the other hand,
        it is obvious that
        $j$ is formally étale,
        and therefore
        $\theta_o=\theta_o^0\circ j$
        is formally étale.
    \end{rk}
    Let $\xi$ be a smooth point on $\circX$ in usual sense.
    Without losing generality, suppose that $\xi$ lies in the $x\not=0$ patch $U\subset\circX$.
    Then the embedding of the underlying schemes $U\hookrightarrow \circX$ induces the log structure on $U$ via
    $$
    \begin{tikzcd}
    	{\bC[y^{\pm1}]} & {\bC[x,y]/(xy)} \arrow[l]  \\
    	& \bN^2 \arrow[lu] \arrow[u]
    \end{tikzcd}
    $$
    where the horizontal map is given by $x\mapsto0$, $y\mapsto y$.
    Thus the above log ring $\bN^2\to\bC[y^{\pm1}]$
    is given by
    $(m,n)\mapsto 0^my^n$,
    which implies that $$\overline{M}_\xi
    \simeq\overline{\bN^2\oplus_{\bN}(\bC[y^{\pm1}])^*}
    \simeq\bN.
    $$
    So the fibers of the characteristic sheaf $\overline{M}_X$ at smooth (in usual sense) points 
    $\infty_1$ and $\infty_2$ are both $\bN$,
    and the log structures near $\infty_1$ and $\infty_2$
    are given by charts $\bN^2\to\bC[x^{-1}]$, $(m,n)\mapsto 0^mx^{-n}$
    and
    $\bN^2\to\bC[y^{-1}]$, $(m,n)\mapsto y^{-m}0^n$
    respectively.
    
    In summary,
    locally the log structure on $X$ can be described as
    $\bN\to\bC[x]$
    (resp. $\bN\to\bC[x^{-1}]$,
    $\bN\to\bC[y]$,
    $\bN\to\bC[y^{-1}]$)
    near smooth points
    and 
    $\bN^2\to R$ near the node.

    After clarifying the log structure on $X$,
    we apply Theorem 
    \ref{etalesmooth}
    to discuss the smoothness of $X$ near $o$.
    The structure morphism
    $f:X\to S$ admits a 
    chart 
    $\Delta:\bN\to\bN^2$
    as the diagonal map.
    Since $\circX$ is a neighborhood of $o$;
    the kernel and torsion part of the cokernel of $\Delta$ are both $0$,
    conditions (1) and (2) in the above theorem are satisfied.
    For condition (3),
    the morphism of schemes
    $\ul{\circX}\to
    \ul{S}\times_{\Spec \bC[\bN]}\Spec \bC[\bN^2]$ corresponds to the ring homomorphism
    $$
    b_{\theta}:\bC\otimes_{\bC[t]}\bC[x,y]
    \to
    R=\bC[x,y]/(xy)$$
    which is induced by $\bC$-homomorphisms:
    \\
    1) $\bC[t]\to\bC$, $t\mapsto 0$;
    \\
    2) $\bC[t]\mapsto \bC[x,y]$,
    $t\mapsto xy$;
    \\
    3) $\bC\hookrightarrow R$;
    \\
    4) $\bC[x,y]\to R$ the quotient map.
    \\
    Let $A$ be a $\bC$-algebra with
    an ideal $J$ of square zero.
    Suppose there is a commutative solid diagram
    $$
    \begin{tikzcd}
    	{\bC\otimes_{\bC[t]}\bC[x,y]} \arrow[d, "f"'] \arrow[r, "b_{\theta}"] & R \arrow[d, "g"] \arrow[ld, "h"', dashed] \\
    	A \arrow[r, "p"]                                                      & A/J                                      
    \end{tikzcd}
    $$
    where $p$ is the quotient map.
    If we define $h:R\to A$
    by
    $h(x)=f(1\otimes x)$,
    $h(y)=f(1\otimes y)$,
    then it is straightforward to check that $h$ makes the diagram commutative.
    Thus by definition $b_\theta$
    is formally smooth.
    Therefore the structure morphism 
    $\circX\to S$ is smooth as its underlying morphism of schemes is of finite presentation.
    Moreover,
    smoothness at $\infty_{1,2}$
    is trivial from the smoothness
    of the morphism of underlying schemes.
    Hence $X\to S$ is smooth and thus is a log curve.

    Let $V$ be a conformal vertex algebra bounded from below, 
    and let $M_1$ (resp. $M_2$) be an inserted $V$-module at $\infty_1$ (resp. $\infty_2$)
    on which $L_0^{M_1}$ (resp. $L_0^{M_2}$)
    acts with integral eigenvalues.
    Define a local coordinate $z_1$ (resp. $z_2$) at $\infty_1$ (resp. $\infty_2$) by 
    \begin{equation}
    	\label{coorz}
    	\begin{tikzcd}
    		{\bC[x^{-1}]} \arrow[r, "z_1^\sharp"] & {\bC[[t]]}      &  & x^{-m}0^n \arrow[r, maps to]               & t^m                        \\
    		\bN^2 \arrow[u] \arrow[r, "z_1^\flat=\id_{\bN^2}"]  & \bN^2 \arrow[u] &  & (m,n) \arrow[u, maps to] \arrow[r, maps to] & {(m,n)} \arrow[u, maps to]
    	\end{tikzcd}
    \end{equation}
    (resp. $z_2$ in the quite similar way via replacing $x$ by $y$).
    Therefore for a log differential $\omega=f(x,y)\frac{dx}{x}+g(x,y)\frac{dy}{y}$ over $\circX$, 
    the pull-back of $\omega$ to $D_{1}^\times$ 
    (resp. $D_{2}^\times$) is $(f(t^{-1},0)-g(t^{-1},0))\frac{dt^{-1}}{t^{-1}}$ 
    (resp. $(g(0,t^{-1})-f(0,t^{-1}))\frac{dt^{-1}}{t^{-1}}$).
    
    We give a description of the $\Lie_{\circX}(V)$-action.
    For a $D$-point $p$ on $\circX$ located in the $x\not=0$ patch $U$,
    by the Cohen structure theorem, 
    we have an isomorphism $\Spec\widehat{
    \cO}_p\simeq \Spec\bC[[t]]\simeq\underline{D}^\circ$.
    The log structure on $U$ is induced by that on $X$, whose local chart is $\bN^2\to\cO_X(U)$.
    Since the log structure on $D$ is $\bN^2\to\bC[[t]]$, $(0,1)\mapsto t$, 
    with the (log) coordinate induced by
    $\id_{\bN^2}$,
    it follows that the (log) coordinates on $U$ is the same as that in usual sense (by means of coordinates in \cite{frenkel-benzvi},
    without log structures).
    
    \begin{rk}[\textbf{splitting of $\cV$}, \cite{frenkel-benzvi}-6.5.9]\label{split}
    	Suppose $V$ is $\bZ$-graded and bounded from below,
    	with finite dimensional homogeneous components $V_n$, $n\geq K$ for some $K\in\bZ$.
    	If $A\in V_\Delta$ is a primary vector 
    	(recall Definition \ref{primary}),
    	then $\bC A$ is an $\Autu^0\cO(\bC)$-submodule of $V$.
    	This gives rise to an subbundle 
    	$(\bC A\otimes\pi_*\cO_{\Fra_D^\infty}(\bC))^{\Autu^0\cO(\bC)}$
    	of $\cV_\bC$ over $D$.
    	Each $\rho\in \Autu^0\cO(\bC)$
    	acts on $A$ by means of
    	formula (\ref{action}).
    	Since $A$ is primary we see that
    	$\rho(z)^{-1}A=(\partial_z\rho(z))^\Delta A$.
    	And thus the coordinate transformation of the subbundle is of the same form with that
    	$(dz)^{-\Delta}=(\rho'(z))^\Delta(d\rho(z))^{-\Delta}$
    	of $\Omega_{D}^{-\Delta}$.   
    	On the other hand,
    	observe that
    	$V_{\leq m}=\oplus_{n=K}^m V_n$
    	is an $\Autu^0\cO(\bC)$-submodule of $V$
    	and there is an exact sequence of vector spaces
    	$$0\to
    	V_{\leq (m-1)}\to
    	V_{\leq m}\to
    	V_m\to0.$$
    	We note that $V_m$ is not preserved by $\Autu^0\cO(\bC)$.
    	But if we consider $V_m$ as a quotient of
    	$V_{\leq m}$
    	by $V_{\leq (m-1)}$,
    	then the above sequence is exact as $\Autu^0\cO(\bC)$-modules,
    	and $V_m$ is a direct sum of one-dimensional representations of $\Autu^0\cO(\bC)$.
    	The associated exact sequence of bundles over $D$:
    	$$0\to
    	\cV_{\leq (m-1)}\to
    	\cV_{\leq m}\to
    	\cV_m\to0$$
    	thus gives rise to an isomorphism of vector bundles
    	$$\cV_m\simeq (\Omega_{D}^{-m})^{\oplus\dim V_m}.$$
    	Therefore as vector bundles,
    	we have a splitting $$\cV_\bC\simeq \bigoplus_{n\geq K}(\Omega_{D}^{-n})^{\oplus\dim V_n}.$$
    \end{rk}
    
    From the above remark,
    we have a splitting $$\cV_{\infty_1}(D)\simeq
    \bigoplus_{n} V_n\otimes\Omega_{D}^{-n}(D)$$
    by $v\otimes t^{-1}\mapsto
    v\otimes(dt^{-1})^{-m}$ for $v\in V_m$.
	Suppose that $f,g\in\bC[x,y]/(xy)\simeq\cO_X(\circX)$ 
	are of the form
	$$f=a_0+a_1x+a_2x^2+...+b_1y+b_2y^2+...,$$
	$$g=a_0'+a_1'x+a_2'x^2+...+b_1'y+b_2'y^2+...$$
	Then by fixing coordinates $z_1$ and $z_2$
	from (\ref{coorz}),
	the pull-backs of $f\frac{dx}{x}+g\frac{dy}{y}\in\Omega_X^1(\circX)$
	in $\Gamma(D^\times_1,\Omega_{X}^1)$ 
	and $\Gamma(D^\times_2,\Omega_{X}^1)$
	are
	$$
	(a_0-a_0'+(a_1-a_1')t^{-1}+(a_2-a_2')t^{-2}+...)\frac{dt^{-1}}{t^{-1}}$$ 
	$$(a_0'-a_0+(b_1'-b_1)t^{-1}+(b_2'-b_2)t^{-2}+...)\frac{dt^{-1}}{t^{-1}}
	$$
	respectively
	since $dx/x=-dy/y$.
	Therefore the space of coinvariants associated to $\circX$ is isomorphic to
	$$V\otimes V/\Lie_{\circX}(V)\cdot (V\otimes V)$$
	where the action of $\mu\in \Lie_{\circX}(V)$ is given by
	\begin{equation}\label{exaction}
		\mu\cdot(u\otimes v)=\cY^\vee_{\infty_1}(\mu|_{D^\times_{1}})u\otimes v+u\otimes\cY^\vee_{\infty_2}(\mu|_{D^\times_{2}})v.
	\end{equation}
	In fact,
	for $v\otimes \left(f\frac{dx}{x}+g\frac{dy}{y}\right)$,
	its pull-back to $D_{i}^\times$, $i=1,2$ is supposed to be a class in
	$$\im\left(
	\cV_{{\infty_i}}(D_i)\otimes\Omega_{D_{i}}^1(D_i)
	\to
	\Gamma(D_i^\times,\cV_\bC)\otimes
	\Gamma(D_i^\times,\Omega_{X}^1)\right),\quad i=1,2,$$
	since $f,g\in\cO_X(\circX)$ are regular over $\circX$.
	Moreover it follows from Remark \ref{split} that
	$$\cV_{{\infty_i}}(D_i)\otimes\Omega_{D_{i}}^1(D_i)
	\simeq
	\bigoplus_{n} V_n\otimes\Omega_{D_i}^{1-n}(D_i)$$
	and thus the associated vertex operator is 
	\begin{equation*}
		\begin{aligned}
			{}&\cY_{\infty_1}^\vee\left(\left(v\otimes\left(f\frac{dx}{x}+g\frac{dy}{y}\right)\right)\Big|_{D^\times_{1}}\right)
			\\
			&=\cY_{\infty_1}^\vee\left(\left(v^{1}\otimes(a_0- a_0')t+v^0\otimes(a_1-a_1')+v^{-1}\otimes(a_2-a_2')t^{-1}+...)dt^{-1}\right)\right)
			\\
			&=(a_0-a_0')v^1_{-1}+(a_1-a_1')v_0^0+(a_2-a_2')v_{1}^{-1}+...
		\end{aligned}
	\end{equation*}
    for some homogeneous $v^i\in V_i$, $i\in\bZ_{\leq1}$
    such that
    $v=\sum_i v^i$ is a finite sum.
    For $a\in V_k$, $k\geq0$, based on the splitting of $\cV_{\infty_1}(D_1)$,  we have that
    $$a\otimes (dt^{-1})^{-k}t^{-k-l}\frac{dt^{-1}}{t^{-1}}=a\otimes t^{1-k-l}(dt^{-1})^{1-k} \stackrel{\cY_{\infty_1}^\vee}{\longmapsto}a_{l+k-1}$$
    The above $t^{-k-l}$ comes from a restriction of regular function (series) on $D^\circ$,
    and thus $l+k\geq0$.
    The action induced by $\Lie_{\circX}(V)$ near $\infty_1$ is then a collection of operators
    $$\{a_{l+k-1}:a\in V_k, l+k\geq0\}$$
    such that each $a_{l+k-1}$ is of degree $-l$.
    Roughly the same,
    it follows that
    $$\cY_{\infty_2}^\vee\left(\left(v\otimes \left(f\frac{dx}{x}+g\frac{dy}{y}\right)\right)\Big|_{D^\times_{2}}\right)
    =(a_0'-a_0)v^1_{-1}+(b_1'-b_1)v_0^0+(b_2'-b_2)v_{1}^{-1}+...$$
    The action induced by $\Lie_{\circX}(V)$ near $\infty_2$
    is of the same form with that near $\infty_1$.
    Comparing the coefficients, we see that the condition $1-k-l=1$ implies
    $k=-l$, 
    and thus the operators 
    induced by the $v^1_{-1}$-terms are $a_{-1}\otimes\id-\id\otimes a_{-1}$
    for homogeneous $a$.
    Therefore the $\Lie_\circX(V)$-action in (\ref{exaction}) on $V\otimes V$ is given by
    $$\cL:=\{a_{-1}\otimes\id-\id\otimes a_{-1}\}\cup
    \{ca_m\otimes\id+c'\id\otimes a_n: m,n\geq0, c,c'\in\bC\}$$
    where $a$ ranges over all homogeneous vectors in $V$.
	
	\begin{thm}
		The space of coinvariants associated to $(X,(\infty_1,\infty_2),(V,V))$ is the zero vector space.
	\end{thm}

    \begin{proof}
    	From the above, 
    	it follows that
    	the coinvariant space is isomorphic to
    	$$V\otimes V/\cL\cdot(V\otimes V).$$
    	Let $c\not=0$, $c'=0$, $a=\omega$, $m=1$.
    	Then we obtain a nonzero operator 
    	$cL_0\otimes \id$.
    	We have an isomorphism
    	$V\otimes V\simeq\oplus_{n\geq K}(V_n\otimes V)$,
    	and let $A^n\otimes B\in V_n\otimes V$.
    	Then it is easy to see that
    	$$A^n\otimes B=\left(\frac{1}{n}L_0\otimes\id\right)
    	(A^n\otimes B).$$
    	This shows that the action of $\cL$ is transitive on 
    	$V\otimes V$,
    	which completes the proof.
    \end{proof}

    \begin{rk}
    	The above theorem seems to be against intuition.
    	Over the complex projective line $\bP^1$, 
    	it is known that (\cite{frenkel-benzvi}-10.4.1)
    	$$\dim C_V(\bP^1,(x,y),(V,V))=
    	\dim C_V(\bP^1,x,V)=1.
    	$$ 
    	The nodal curve $\underline{X}$ corresponds to a boundary point in the family $\{X_t\}$ of genus-$0$ smooth curves in $\bP^2$ determined by the homogeneous equations $xy=tz^2$, parameterized by $t\in\bR_{>0}$.
    	Intuitively we may consider $\underline{X}$ as ``
    	$\lim_{t\to0}X_t$".
    	For each $t\not=0$,
    	the quotient space $X_t/\{\pm1\}$ 
    	($\pm1$ acts on $z$ by scalar multiplication)
    	is isomorphic to $\bP^1$,
    	and thus the conformal block spaces associated to the same vertex algebra $V$ over these curves are of the same dimension.
    	In general, the structure of the associated bundle over $\underline{X}$ is believed to be more complicated than those over $X_t$, 
    	but it turns to be on the opposite side.
    	The reason is that for any differential form on $\bP^1-\{x,y\}$,
    	its restrictions to
    	the punctured discs $D_x^\times$ and $D_y^\times$ are closely related.
    	In contrast, for a log differential $fdx/x+gdy/y$ over $X$, 
    	its restrictions on different patches turns to be essentially unrelated except for the constant terms,
    	due to the log structure on the base field 
    	$\bN\to \bC$.
    	Consequently the resulting $\Lie_{\circX}(V)$-action is much bigger than expected  to be.
    \end{rk}

    \subsection{Case II: base field $0\to\bC$}
    
    We use the same notation as in the preceding case,
    and additionally set $S_0=\Spec(0\to\bC)$.
    To distinguish the base change from that via 
    $S=\Spec(\bN\to\bC)$,
    for the discs $D$, $D^\circ$ and $D^\times$,
    we denote by
    $D_0$, $D_0^\circ$, $D^\times_0$
    their base changes via $S_0$ respectively.
    Concretely,
    \begin{equation*}
    	\begin{aligned}
    		D_0&=\Spf(\bN\to \bC[[t]]),\quad n\mapsto t^n;\\
    		D^\circ_0&=\Spec(\bN\to \bC[[t]]),\quad n\mapsto t^n;\\
    		D^\times_0&=\Spec(0\to \bC((t))).
    	\end{aligned}
    \end{equation*}

    Then the $S_0$-log scheme structure of $X$ is given by
    $$
    \begin{tikzcd}
    	\bC \arrow[r]          & R               &  & 1 \arrow[r, maps to]                    & 1                          \\
    	0 \arrow[r] \arrow[u] & \bN^2 \arrow[u] &  & 0 \arrow[r, maps to] \arrow[u, maps to] & {(0,0)} \arrow[u, maps to]
    \end{tikzcd}
    $$
    From the construction of the module of log differentials in Section \ref{logsch},
    the $R$-module
    $\Omega_{\circX/S_0}^1(\circX)$
    is generated by
    $\frac{dx}{x}$ and $\frac{dy}{y}$.
    We note that the relation 
    $\frac{dx}{x}+\frac{dy}{y}=0$
    in Case I
    does not hold in Case II.
	
	However,
	it is hardly possible to define a coordinate at $o$ in Case II,
	for the following reason.
	Consider the following diagram 
	(not necessarily commutative)
	associated to a coordinate
	(if it were to exist):
	$$
	\begin{tikzcd}
		x^my^n                     & R \arrow[r]               & {\bC[[t]]}    & t^m                  \\
		{(m,n)} \arrow[u, maps to] & \bN^2 \arrow[r] \arrow[u] & \bN \arrow[u] & m \arrow[u, maps to]
	\end{tikzcd}
    $$
	Suppose that the bottom horizontal morphism of monoids
	$\bN^2\to \bN$
	is given by
	$(1,0)\mapsto p$, $(0,1)\mapsto q$,
	and then
	$(m,n)\mapsto mp+nq$. 
    If the above square is commutative,
    then
    $x^my^n\mapsto t^{mp+nq}$.
    However,
    it is easy to see that
    $mn\not=0$
    implies
    $x^my^n=0$ while
    $t^{mp+nq}\not=0$.
    Hence we cannot expect to define a coordinate if the log structure on the base field is trivial.
    This also suggests that a non-trivial additional structure on the base field is necessary when a node appears,
    which is precisely why we use log geometry.
	
\nocite{*}
 \bibliographystyle{plain}
 \bibliography{valog}

\end{document}